\documentclass[reqno]{amsart}
\usepackage{amsmath, amssymb, amsthm, epsfig}
\usepackage{hyperref, latexsym}
\usepackage{url} 
\def\today{\ifcase\month\or
  January\or February\or March\or April\or May\or June\or
  July\or August\or September\or October\or November\or December\fi
  \space\number\day, \number\year}

\DeclareMathOperator{\sgn}{\mathrm{sgn}}

 \newtheorem{theorem}{Theorem}[section]
 \newtheorem{lemma}[theorem]{Lemma}
 
 \newtheorem{corollary}[theorem]{Corollary}
 \theoremstyle{definition}

 \theoremstyle{remark}
 \newtheorem{remark}[theorem]{Remark}
 \newcommand{\mc}{\mathcal}
 \newcommand{\A}{\mc{A}}

 \newcommand{\G}{\mc{G}}
 \newcommand{\sH}{\mc{H}}

 \newcommand{\K}{\mc{K}}

 \newcommand{\rr}{\mc{R}}

 \newcommand{\C}{\mathbb{C}}
 \newcommand{\R}{\mathbb{R}}

 \newcommand{\Z}{\mathbb{Z}}
 \newcommand{\csch}{\mathrm{csch}}
 \newcommand{\tF}{\widehat{F}}
 \newcommand{\tG}{\widetilde{G}}
 \newcommand{\tH}{\widetilde{H}}
 \newcommand{\ug}{\widehat{g}}
 \newcommand{\wg}{\widetilde{g}}
 \newcommand{\uh}{\widehat{h}}
 \newcommand{\wh}{\widetilde{h}}
 \newcommand{\wl}{\widetilde{l}}
 \newcommand{\tL}{\widehat{L}}
 \newcommand{\wm}{\widetilde{m}}
 \newcommand{\tM}{\widehat{M}}
 
 \newcommand{\tq}{\widehat{q}}
 
 \newcommand{\tu}{\widehat{u}}
 \newcommand{\tv}{\widehat{v}}
 \newcommand{\ba}{\boldsymbol{a}}
 
 \newcommand{\bx}{\boldsymbol{x}}
 
 \newcommand{\h}{\frac12}
 \newcommand{\hh}{\tfrac12}
 \newcommand{\ds}{\text{\rm d}s}
 \newcommand{\dt}{\text{\rm d}t}
 \newcommand{\du}{\text{\rm d}u}
 \newcommand{\dv}{\text{\rm d}v}
 \newcommand{\dw}{\text{\rm d}w}
 \newcommand{\dx}{\text{\rm d}x}
 \newcommand{\dy}{\text{\rm d}y}
 \newcommand{\dl}{\text{\rm d}\lambda}
 \newcommand{\dmu}{\text{\rm d}\mu(\lambda)}
 \newcommand{\dnu}{\text{\rm d}\nu(\lambda)}
 
 \newcommand{\dnnu}{\text{\rm d}\nu_n(\lambda)}

\begin{document}

\title[Extremal functions]{Some extremal functions\\in Fourier analysis, II}
\author[Carneiro and Vaaler]{Emanuel Carneiro* and Jeffrey D. Vaaler**}
\thanks{*Research supported by CAPES/FULBRIGHT grant BEX 1710-04-4.}
\thanks{**Research supported by the National Science Foundation, DMS-06-03282.}
\date{\today}
\address{Department of Mathematics, University of Texas at Austin, Austin, TX 78712-1082.}
\email{ecarneiro@math.utexas.edu}
\address{Department of Mathematics, University of Texas at Austin, Austin, TX 78712-1082.}
\email{vaaler@math.utexas.edu}
\allowdisplaybreaks
\numberwithin{equation}{section}

\begin{abstract} We obtain extremal majorants and minorants of exponential type for a class of 
even functions on $\R$ which includes $\log |x|$ and $|x|^\alpha$, where $-1 < \alpha < 1$.  We also give periodic
versions of these results in which the majorants and minorants are trigonometric polynomials of bounded degree.
As applications we obtain optimal estimates for certain Hermitian forms, which include discrete analogues of 
the one dimensional Hardy-Littlewood-Sobolev inequalities.  A further application provides
an Erd\"{o}s-Tur\'{a}n-type inequality that estimates the sup norm of algebraic polynomials on the unit disc
in terms of power sums in the roots of the polynomials.
\end{abstract}

\maketitle
\section{Introduction}

In this paper we consider the following extremal problem.  Let $f:\R\rightarrow \R$ be a given function. 
Determine real entire functions $G:\C\rightarrow\C$ and $H:\C\rightarrow\C$ such that $G$ and $H$ have exponential 
type at most $2\pi$, and satisfy the inequality 
\begin{equation}\label{intro0}
G(x) \le f(x) \le H(x)
\end{equation} 
for all real $x$.  And among such functions $G$ and $H$, determine those for which the integrals
\begin{equation}\label{intro1}
\int_{-\infty}^{\infty} \left\{f(x) - G(x)\right\}\ \dx\quad\text{and}
			\quad \int_{-\infty}^{\infty} \left\{H(x) - f(x)\right\}\ \dx
\end{equation}
are minimized.  By a real entire function we understand an entire function that takes 
real values at points of $\R$.  

In the special case $f(x) = \sgn(x)$, an explicit solution to this problem was found in the 1930's by A. Beurling, but 
his results were not published at the time of their discovery.  Later, Beurling's solution was rediscovered by 
A. Selberg, who recognized its importance in connection with the large sieve inequality of analytic number theory.  
In particular, Selberg observed that Beurling's function could be used to majorize and minorize the function
\begin{equation}\label{intro3}
\hh\sgn(x - a) + \hh\sgn(b - x) = \begin{cases}  1& \text{if $a < x < b$,}\\
					  \hh& \text{if $x = a$ or $x = b$,}\\
					  0& \text{if $x < a$ or $b < x$,}\end{cases}
\end{equation}
where $a < b$.  Of course, this function is essentially the characteristic function of the interval with
endpoints $a$ and $b$.  The functions that majorize and minorize (\ref{intro3}) are real entire functions
of exponential type at most $2\pi$, but in applications it is often useful to exploit the fact that their
Fourier transforms are continuous functions supported on the interval $[-1, 1]$. 
An account of these functions, the history of their discovery, and many applications
can be found in \cite{GV}, \cite{GV2}, \cite{M}, \cite{S}, \cite{V}, and \cite{V2}.  Further examples have been 
given by F. Littmann \cite{Lit}, \cite{Lit2}, and extensions of the problem to several variables are considered 
in \cite{BMV}, \cite{HV}, and \cite{LV}.

Let $\lambda$ be a positive real parameter.  Define entire functions $z\mapsto L(\lambda,z)$ and 
$z\mapsto M(\lambda,z)$ by
\begin{equation}\label{intro4}
L(\lambda,z) = \left(\dfrac{\cos\pi z}{\pi}\right)^2\Bigg\{\sum_{k\in\Z} \dfrac{e^{-\lambda|k + \h|}}{(z - k - \h)^2}
	- \lambda\sum_{l\in\Z} \dfrac{\sgn(l+\h)e^{-\lambda|l + \h|}} {(z - l - \h)} \Bigg\},
\end{equation}
and
\begin{equation}\label{intro5}
M(\lambda,z) = \left(\dfrac{\sin \pi z}{\pi}\right)^2\Bigg\{\sum_{k\in \Z} \dfrac{e^{-\lambda|k|}}{(z - k)^2} - 
	\lambda\sum_{l\in\Z} \dfrac{\sgn(l)e^{-\lambda|l|}}{(z - l)}\Bigg\}.
\end{equation}
In \cite{GV} it was shown that both $z\mapsto L(\lambda,z)$  and $z\mapsto M(\lambda,z)$ are real entire 
functions of exponential type $2\pi$, they are bounded and integrable on $\R$, and they satisfy the inequality
\begin{equation}\label{intro6}
L(\lambda, x) \leq e^{-\lambda |x|} \leq M(\lambda, x)
\end{equation}
for all real $x$.  Moreover, for each positive value of $\lambda$ the functions $z\mapsto L(\lambda,z)$ and 
$z\mapsto M(\lambda,z)$ are the unique extremal functions for the problem of minimizing the integrals 
(\ref{intro1}).  That is, the values of the two integrals
\begin{equation}\label{intro7}
\int_{-\infty}^{\infty} \left\{ e^{-\lambda |x|} - L(\lambda,x) \right\}\ \dx 
		= \tfrac{2}{\lambda} - \csch\left(\tfrac{\lambda}{2}\right),
\end{equation}
and 
\begin{equation}\label{intro8}
\int_{-\infty}^{\infty} \left\{M(\lambda,x) - e^{-\lambda |x|} \right\}\ \dx 
		= \coth\left(\tfrac{\lambda}{2}\right) - \tfrac{2}{\lambda},
\end{equation}
are both minimal.  It was also shown in \cite{GV} that the Fourier transforms
\begin{equation}\label{intro8.2}
\tL(\lambda,t) = \int_{-\infty}^{\infty}L(\lambda,x)e(-tx)\ \dx\quad\text{and}\quad
	\tM(\lambda,t) = \int_{-\infty}^{\infty}M(\lambda,x)e(-tx)\ \dx
\end{equation}
are continuous functions of the real variable $t$ supported on the interval $[-1,1]$.  Here we write 
$e(z) = e^{2\pi i z}$.  Both Fourier transforms in (\ref{intro8.2}) are nonnegative functions of $t$ 
and are given explicitly here in Lemma \ref{lem3.3}.  

If $\mu$ is a suitable measure defined on the Borel subsets of $(0,\infty)$, then one might hope to show that
\begin{equation}\label{extra1}
z\mapsto \int_0^{\infty} L(\lambda, z)\ \dmu\quad\text{and}\quad z\mapsto \int_0^{\infty} M(\lambda, z)\ \dmu
\end{equation}  
both define real entire functions of $z$ with exponential type at most $2\pi$.  If this is so then they clearly
satisfy the inequality
\begin{equation}\label{extra2}
\int_0^{\infty} L(\lambda, x)\ \dmu \le \int_0^{\infty} e^{-\lambda |x|}\ \dmu \le
 			\int_0^{\infty} M(\lambda, x)\ \dmu		
\end{equation}
for all real $x$.  In this case one may also hope to show that these real entire functions are extremal with respect to
the problem of majorizing and minorizing the function
\begin{equation*}\label{extra3}
x\mapsto \int_0^{\infty} e^{-\lambda |x|}\ \dmu.
\end{equation*}
In fact such a result was obtained in \cite[Theorem 9]{GV}, but only under the restrictive hypothesis that
\begin{equation}\label{extra4}
\int_0^{\infty} \frac{\lambda + 1}{\lambda}\ \dmu < \infty.
\end{equation}
In the present paper we solve the extremal problem
for a wider class of measures.  By making special choices for $\mu$, we are able to give explicit
solutions to the extremal problem for such examples as $x\mapsto \log |x|$ and $x\mapsto |x|^{\alpha}$, where
$-1 < \alpha < 1$.  We now describe these results.

Let $\mu$ be a measure defined on the Borel subsets of $(0,\infty)$ such that
\begin{equation}\label{intro31}
0 < \int_0^{\infty} \frac{\lambda}{\lambda^2 + 1}\ \dmu < \infty.
\end{equation}
It follows from (\ref{intro31}) that for $x\not= 0$ the function
\begin{equation*}\label{intro32}
\lambda\mapsto e^{-\lambda |x|} - e^{-\lambda}
\end{equation*}
is integrable on $(0,\infty)$ with respect to $\mu$.
We define $f_{\mu}:\R\rightarrow \R\cup\{\infty\}$ by
\begin{equation}\label{intro33}
f_{\mu}(x) = \int_0^{\infty} \big\{e^{-\lambda |x|} - e^{-\lambda}\big\}\ \dmu,
\end{equation}
where 
\begin{equation*}\label{intro34}
f_{\mu}(0) = \int_0^{\infty} \big\{1 - e^{-\lambda}\}\ \dmu
\end{equation*}
may take the value $\infty$.  Clearly $f_{\mu}(x)$ is infinitely differentiable at each 
real number $x\not= 0$.  In particular, we find that
\begin{equation*}\label{intro35}
f^{\prime}_{\mu}(x) = -\sgn(x)\int_0^{\infty} \lambda e^{-\lambda |x|}\ \dmu
\end{equation*}
for all $x\not= 0$.  Using $f_{\mu}$ and $f^{\prime}_{\mu}$, we define $G_{\mu}:\C\rightarrow\C$ by
\begin{equation}\label{intro36}
G_{\mu}(z) = \lim_{N\rightarrow\infty} \left(\frac{\cos \pi z}{\pi}\right)^2 
	\left\{\sum_{n=-N}^{N+1} \frac{f_{\mu}(n-\hh)}{(z - n + \hh)^2} 
		+ \sum_{n=-N}^{N+1}\frac{f^{\prime}_{\mu}(n-\hh)}{(z - n + \hh)}\right\}.
\end{equation}
We will show that the limit on the right of (\ref{intro36}) converges uniformly on compact subsets of $\C$ and therefore
defines $G_{\mu}(z)$ as a real entire function.  Then it is easy to check that $G_{\mu}$ interpolates the values of 
$f_{\mu}$ and $f^{\prime}_{\mu}$ at real numbers $x$ such that $x + \h$ is an integer.  That is, the system of identities
\begin{equation}\label{intro37}
G_{\mu}(n - \hh) = f_{\mu}(n - \hh)\quad\text{and}\quad G^{\prime}_{\mu}(n - \hh) = f^{\prime}_{\mu}(n - \hh)
\end{equation}
holds for each integer $n$.  

Because $f_{\mu}(0)$ may take the value $\infty$, there can be no question of majorizing $f_{\mu}(x)$ by a real entire
function.  However, we will prove that the real entire function $G_{\mu}(z)$ minorizes $f_{\mu}(x)$ on $\R$, and 
satisfies the following extremal property.

\begin{theorem}\label{thm1.1}  Assume that the measure $\mu$ satisfies {\rm (\ref{intro31})}.  
\begin{itemize}
\item[(i)]  The real entire function $G_{\mu}(z)$ defined by {\rm (\ref{intro36})} has exponential type at most $2\pi$.
\item[(ii)]  For real $x\not= 0$ the function
\begin{equation*}\label{intro40}
\lambda \mapsto e^{-\lambda |x|} - L(\lambda, x)
\end{equation*}
is nonnegative and integrable on $(0, \infty)$ with respect to $\mu$. 
\item[(iii)]  For all real $x$ we have
\begin{equation}\label{intro41}
0 \le f_{\mu}(x) - G_{\mu}(x) = \int_0^{\infty} \big\{e^{-\lambda |x|} - L(\lambda, x)\big\}\ \dmu.
\end{equation}
\item[(iv)]  The nonnegative function $x\mapsto f_{\mu}(x) - G_{\mu}(x)$ is integrable on $\R$, and
\begin{equation}\label{intro42}
\int_{-\infty}^{\infty} \left\{f_{\mu}(x) - G_{\mu}(x)\right\}\ \dx
		= \int_0^{\infty}\big\{\tfrac{2}{\lambda} - \csch\bigl(\tfrac{\lambda}{2}\bigr)\big\}\ \dmu.
\end{equation}
\item[(v)]  If $t\not= 0$ then
\begin{align}\label{intro43}
\begin{split}
\int_{-\infty}^{\infty} \big\{f_{\mu}(x) - &G_{\mu}(x)\big\}e(-tx)\ \dx \\
	&= \int_0^{\infty} \frac{2\lambda}{\lambda^2 + 4\pi^2 t^2}\ \dmu
		- \int_0^{\infty} \tL\bigl(\lambda, t\bigr)\ \dmu.
\end{split}
\end{align}
\item[(vi)]  If $\tG(z)$ is a real entire function of exponential type at most $2\pi$ such that
\begin{equation*}\label{intro44}
\tG(x) \le f_{\mu}(x)
\end{equation*}
for all real $x$, then
\begin{equation}\label{intro45}
\int_{-\infty}^{\infty} \left\{f_{\mu}(x) - G_{\mu}(x)\right\}\ \dx 
		\le \int_{-\infty}^{\infty} \left\{f_{\mu}(x) - \tG(x)\right\}\ \dx.
\end{equation}
\item[(vii)]  There is equality in the inequality {\rm (\ref{intro45})} if and only if $\tG(z) = G_{\mu}(z)$.
\end{itemize}
\end{theorem}  

Now assume that the measure $\mu$ satisfies the condition
\begin{equation}\label{intro47}
0 < \int_0^{\infty} \frac{\lambda}{\lambda + 1}\ \dmu < \infty,
\end{equation}
which is obviously more restrictive than (\ref{intro31}).  From (\ref{intro47}) we have
\begin{equation*}\label{intro48}
f_{\mu}(x) \le f_{\mu}(0) = \int_0^{\infty} \big\{1 - e^{-\lambda}\}\ \dmu < \infty,
\end{equation*}
for all real $x$.  Thus we may try to determine a real entire function that majorizes $f_{\mu}(x)$ on $\R$.
Toward this end we define $H_{\mu}:\C\rightarrow\C$ by
\begin{equation}\label{intro49}
H_{\mu}(z) = \lim_{N\rightarrow\infty} \left(\frac{\sin \pi z}{\pi}\right)^2 
	\left\{\sum_{|n| \le N} \frac{f_{\mu}(n)}{(z - n)^2} 
		+ \sum_{1 \le |n| \le N}\frac{f^{\prime}_{\mu}(n)}{(z - n)}\right\}.
\end{equation}
Again we will show that the limit on the right of (\ref{intro49}) converges uniformly on compact subsets of $\C$ 
and therefore defines $H_{\mu}(z)$ as a real entire function.  In this case the function $H_{\mu}$ interpolates 
the values of $f_{\mu}$ and $f^{\prime}_{\mu}$ at the nonzero integers.  That is, the identities
\begin{equation*}\label{intro50}
H_{\mu}(n) = f_{\mu}(n)\quad\text{and}\quad H^{\prime}_{\mu}(n) = f^{\prime}(n)
\end{equation*}
hold at each integer $n\not= 0$, and at zero we find that
\begin{equation*}\label{intro51}
H_{\mu}(0) = f_{\mu}(0)\quad\text{and}\quad H^{\prime}_{\mu}(0) = 0.
\end{equation*}

As (\ref{intro47}) is more restrictive than (\ref{intro31}), the function $G_{\mu}(z)$ continues to minorize 
$f_{\mu}(x)$ on $\R$ as described in Theorem \ref{thm1.1}.  We will prove that the real entire function $H_{\mu}(z)$ 
majorizes $f_{\mu}(x)$ on $\R$, and satisfies an analogous extremal property.

\begin{theorem}\label{thm1.2}  Assume that the measure $\mu$ satisfies {\rm (\ref{intro47})}.  
\begin{itemize}
\item[(i)]  The real entire function $H_{\mu}(z)$ defined by {\rm (\ref{intro49})} has exponential type at most $2\pi$.
\item[(ii)]  For all real $x$ the function
\begin{equation*}\label{intro52}
\lambda \mapsto M(\lambda, x) - e^{-\lambda |x|}
\end{equation*}
is nonnegative and integrable on $(0, \infty)$ with respect to $\mu$.
\item[(iii)]  For all real $x$ we have
\begin{equation}\label{intro53}
0 \le H_{\mu}(x) - f_{\mu}(x) = \int_0^{\infty} \big\{M(\lambda, x) - e^{-\lambda |x|}\big\}\ \dmu.
\end{equation}
\item[(iv)]  The nonnegative function $x\mapsto H_{\mu}(x) - f_{\mu}(x)$ is integrable on $\R$, and
\begin{equation}\label{intro54}
\int_{-\infty}^{\infty} \left\{H_{\mu}(x) - f_{\mu}(x)\right\}\ \dx
		= \int_0^{\infty}\big\{\coth\bigl(\tfrac{\lambda}{2}\bigr) - \tfrac{2}{\lambda}\big\}\ \dmu.
\end{equation}
\item[(v)]  If $t\not= 0$ then
\begin{align}\label{intro55}
\begin{split}
\int_{-\infty}^{\infty} \big\{H_{\mu}(x) - &f_{\mu}(x)\big\}e(-tx)\ \dx \\
	&= \int_0^{\infty} \tM\bigl(\lambda, t\bigr)\ \dmu
		- \int_0^{\infty} \frac{2\lambda}{\lambda^2 + 4\pi^2 t^2}\ \dmu.
\end{split}
\end{align}
\item[(vi)]  If $\tH(z)$ is a real entire function of exponential type at most $2\pi$ such that
\begin{equation*}\label{intro56}
f_{\mu}(x) \le \tH(x)
\end{equation*}
for all real $x$, then
\begin{equation}\label{intro57}
\int_0^{\infty} \left\{H_{\mu}(x) - f_{\mu}(x)\right\}\ \dx \le \int_0^{\infty} \left\{\tH(x) - f_{\mu}(x)\right\}\ \dx.
\end{equation}
\item[(vii)]  There is equality in the inequality {\rm (\ref{intro57})} if and only if $\tH(z) = H_{\mu}(z)$.
\end{itemize}
\end{theorem}

The real entire functions $G_{\mu}(z)$ and $H_{\mu}(z)$, which occur in Theorem \ref{thm1.1} and Theorem \ref{thm1.2}, have
exponential type at most $2\pi$.  It is often useful to have results of the same sort in which the majorizing 
and minorizing functions have exponential type at most $2\pi\delta$, where $\delta$ is a positive parameter.  To 
accomplish this we introduce a second measure $\nu$ defined on Borel subsets $E\subseteq (0,\infty)$ by
\begin{equation}\label{intro59}
\nu(E) = \mu(\delta E),
\end{equation}
where
\begin{equation*}
\delta E = \{\delta x: x\in E\}
\end{equation*}
is the dilation of $E$ by $\delta$.  If $\mu$ satisfies (\ref{intro31}) then $\nu$ also satisfies (\ref{intro31}), and 
the two functions $f_{\mu}(x)$ and $f_{\nu}(x)$ are related by the identity
\begin{align}\label{intro60}
\begin{split}
f_{\nu}(x) &= \int_0^{\infty} \big\{e^{-\lambda |x|} - e^{-\lambda}\big\}\ \dnu \\
	   &= \int_0^{\infty} \big\{e^{-\lambda \delta^{-1}|x|} - e^{-\lambda \delta^{-1}}\big\}\ \dmu \\
	   &= \int_0^{\infty} \big\{e^{-\lambda |\delta^{-1}x|} - e^{-\lambda}\big\}\ \dmu
	   		- \int_0^{\infty} \big\{e^{-\lambda \delta^{-1}} - e^{-\lambda}\big\}\ \dmu \\
	   &= f_{\mu}\bigl(\delta^{-1} x\bigr) - f_{\mu}\bigl(\delta^{-1}\bigr).
\end{split}
\end{align}
We apply Theorem \ref{thm1.1} to the functions $f_{\nu}(x)$ and $G_{\nu}(z)$.  Then using (\ref{intro60}) we obtain
corresponding results for the functions 
\begin{equation*}
f_{\mu}(x) - f_{\mu}\bigl(\delta^{-1}\bigr) = f_{\nu}(\delta x)\quad\text{and}\quad G_{\nu}(\delta z),
\end{equation*}
where the entire function $z\mapsto G_{\nu}(\delta z)$ has exponential type at most $2\pi\delta$.  This leads easily to the 
following more general form of Theorem \ref{thm1.1}.  We have only stated those parts which we will use in later applications.

\begin{theorem}\label{thm1.3}  Assume that the measure $\mu$ satisfies {\rm (\ref{intro31})}, and let $\nu$ be
the measure defined by {\rm (\ref{intro59})}, where $\delta$ is a positive parameter.  
\begin{itemize}
\item[(i)]  The real entire function $z\mapsto G_{\nu}(\delta z)$ has exponential type at most $2\pi\delta$.
\item[(ii)]  For real $x\not= 0$ the function
\begin{equation}\label{intro61}
\lambda \mapsto e^{-\lambda |x|} - L\bigl(\delta^{-1}\lambda, \delta x\bigr)
\end{equation}
is nonnegative and integrable on $(0, \infty)$ with respect to $\mu$.  
\item[(iii)]  For all real $x$ we have
\begin{align}\label{intro62}
\begin{split}
0 \le f_{\mu}(x) - f_{\mu}\bigl(\delta^{-1}\bigr)&- G_{\nu}(\delta x) \\
	&= \int_0^{\infty} \big\{e^{-\lambda |x|} - L\bigl(\delta^{-1}\lambda, \delta x\bigr)\big\}\ \dmu.
\end{split}
\end{align}
\item[(iv)]  The nonnegative function $x\mapsto f_{\mu}(x) - f_{\mu}\bigl(\delta^{-1}\bigr) - G_{\nu}(\delta x)$ 
is integrable on $\R$, and
\begin{align}\label{intro63}
\begin{split}
\int_{-\infty}^{\infty} \big\{f_{\mu}(x) - f_{\mu}\bigl(\delta^{-1}\bigr)&- G_{\nu}(\delta x)\big\}\ \dx \\
     &= \int_0^{\infty}\big\{\tfrac{2}{\lambda} 
     		- \tfrac{1}{\delta}\csch\bigl(\tfrac{\lambda}{2\delta}\bigr)\big\}\ \dmu.
\end{split}
\end{align}
\item[(v)]  If $t\not= 0$ then
\begin{align}\label{intro64}
\begin{split}
\int_{-\infty}^{\infty} \big\{&f_{\mu}(x) - f_{\mu}\bigl(\delta^{-1}\bigr) - G_{\nu}(\delta x)\big\}e(-tx)\ \dx \\
	&= \int_0^{\infty} \frac{2\lambda}{\lambda^2 + 4\pi^2 t^2}\ \dmu
		- \delta^{-1} \int_0^{\infty} \tL\bigl(\delta^{-1}\lambda, \delta^{-1}t\bigr)\ \dmu.
\end{split}
\end{align}
\end{itemize}
\end{theorem}  

Here is the analogous result for the problem of majorizing $f_{\mu}(x)$.  This is proved by applying 
Theorem \ref{thm1.2} to the functions $f_{\nu}(x)$ and $H_{\nu}(x)$, and then making the same change of 
variables that occurs in the proof of Theorem \ref{thm1.3}.

\begin{theorem}\label{thm1.4}  Assume that the measure $\mu$ satisfies {\rm (\ref{intro47})}, and let $\nu$ be the
measure defined by {\rm (\ref{intro59})}, where $\delta$ is a positive parameter.  
\begin{itemize}
\item[(i)]  The real entire function $z\mapsto H_{\nu}(\delta z)$ defined by {\rm (\ref{intro49})} has exponential 
type at most $2\pi\delta$.
\item[(ii)]  For all real $x$ the function
\begin{equation}\label{intro70}
\lambda \mapsto M\bigl(\delta^{-1}\lambda, \delta x\bigr) - e^{-\lambda |x|}
\end{equation}
is nonnegative and integrable on $(0, \infty)$ with respect to $\mu$.
\item[(iii)]  For all real $x$ we have
\begin{align}\label{intro71}
\begin{split}
0 \le H_{\nu}(\delta x) + f_{\mu}\bigl(\delta^{-1}\bigr)&- f_{\mu}(x) \\
	&= \int_0^{\infty} \big\{M\bigl(\delta^{-1}\lambda, \delta x\bigr) - e^{-\lambda |x|}\big\}\ \dmu.
\end{split}
\end{align}
\item[(iv)]  The nonnegative function $x\mapsto H_{\nu}(\delta x) + f\bigl(\delta^{-1}\bigr) - f_{\mu}(x)$ is 
integrable on $\R$, and
\begin{align}\label{intro72}
\begin{split}
\int_{-\infty}^{\infty} \big\{H_{\nu}(\delta x) + f\bigl(\delta^{-1}\bigr)&- f_{\mu}(x)\big\}\ \dx \\
		&= \int_0^{\infty}\big\{\tfrac{1}{\delta}\coth\bigl(\tfrac{\lambda}{2\delta}\bigr) 
			- \tfrac{2}{\lambda}\big\}\ \dmu.
\end{split}
\end{align}
\item[(v)]  If $t\not= 0$ then
\begin{align}\label{intro73}
\begin{split}
\int_{-\infty}^{\infty} \big\{&H_{\nu}(\delta x) + f\bigl(\delta^{-1}\bigr) - f_{\mu}(x)\big\}e(-tx)\ \dx \\
	&= \delta^{-1}\int_0^{\infty} \tM\bigl(\delta^{-1}\lambda, \delta^{-1} t\bigr)\ \dmu
		- \int_0^{\infty} \frac{2\lambda}{\lambda^2 + 4\pi^2 t^2}\ \dmu.
\end{split}
\end{align}
\end{itemize}
\end{theorem}

\noindent We note that each of the functions
\begin{equation*}\label{intro75}
t\mapsto \delta^{-1} \int_0^{\infty} \tL\bigl(\delta^{-1}\lambda, \delta^{-1}t\bigr)\ \dmu\quad\text{and}\quad
	t\mapsto \delta^{-1} \int_0^{\infty} \tM\bigl(\delta^{-1}\lambda, \delta^{-1}t\bigr)\ \dmu,
\end{equation*}
which occur in the statement of Theorem \ref{thm1.3} and Theorem \ref{thm1.4}, respectively, are continuous on 
$\R$ and supported on $[-\delta, \delta]$. 

As an example to illustrate how these results can be applied, we consider the problem of majorizing the function 
$x\mapsto \log |x|$ by a real entire function $z\mapsto U(z)$ of exponential type at most $2\pi$.  
This special case was first obtained by M.~Lerma \cite{L}.  We select $\mu$ to be a Haar measure on the 
multiplicative group $(0, \infty)$, so that 
\begin{equation}\label{intro80}
\mu(E) = \int_E \lambda^{-1}\ \dl
\end{equation}
for all Borel subsets $E$.  For this measure $\mu$ we find that
\begin{equation*}\label{intro81}
f_{\mu}(x) = - \log |x|.
\end{equation*}
We apply Theorem \ref{thm1.1} with $U(z) = - G_{\mu}(z)$.  Thus the function $U(z)$ is given by
\begin{align}\label{intro82}
\begin{split}
U(z) = \lim_{N\rightarrow\infty} \left(\frac{\cos \pi z}{\pi}\right)^2 
			\bigg\{\sum_{n=-N}^{N+1}&\frac{\log |n-\hh|}{(z - n + \hh)^2} \\
			&+ \sum_{n=-N}^{N+1}\frac{1}{(n - \hh)(z - n + \hh)}\bigg\},
\end{split}
\end{align}
where the limit converges uniformly on compact subsets of $\C$.  From Theorem \ref{thm1.1} we conclude that 
$U(z)$ is a real entire function of exponential type at most $2\pi$, and the inequality
\begin{equation}\label{intro83}
\log |x| \le U(x)
\end{equation}
holds for all real $x$.  From (\ref{intro42}) we get
\begin{equation}\label{intro84}
\int_{-\infty}^{\infty} \left\{U(x) - \log |x|\right\}\ \dx = \log 2.
\end{equation}
Using (\ref{intro43}), for $t\not= 0$ the Fourier transform is
\begin{equation}\label{intro85}
\int_{-\infty}^{\infty} \big\{U(x) - \log |x|\big\}e(-tx)\ \dx \\
	= \bigl(2|t|\bigr)^{-1} - \int_0^{\infty} \tL\bigl(\lambda, t\bigr)\lambda^{-1}\ \dl,
\end{equation}
where $\tL(\lambda, t)$ is given explicitly in Lemma \ref{lem3.3}.  Then Corollary \ref{cor3.1} implies that
\begin{equation}\label{intro86}
0 \le \int_{-\infty}^{\infty} \big\{U(x) - \log |x|\big\}e(-tx)\ \dx \le \bigl(2|t|\bigr)^{-1}
\end{equation}
for all real $t\not= 0$, and there is equality in the inequality on the right of (\ref{intro86}) for $1 \le |t|$.  
Further results and numerical approximations for the function $U(z)$ are given in \cite{L}.

In a similar manner Theorem \ref{thm1.3} can be applied to determine an entire function
of exponential type at most $2\pi\delta$ that majorizes $x\mapsto \log |x|$.  Alternatively, the functional
equation for the logarithm allows us to accomplish this directly.  Clearly the real entire function
\begin{equation*}\label{intro87}
z\mapsto -\log \delta + U(\delta z)
\end{equation*}
has exponential type at most $2\pi\delta$, majorizes $x\mapsto \log |x|$ on $\R$, and satisfies
\begin{equation}\label{intro88}
\int_{-\infty}^{\infty} \big\{-\log \delta + U(\delta x) - \log |x|\big\}\ \dx = \frac{\log 2}{\delta}.
\end{equation}

Another interesting application arises when we choose measures $\mu_{\sigma}$ such that 
\begin{equation}\label{Intro26.1}
\mu_{\sigma}(E) = \int_E \lambda^{-\sigma}\ \dl,
\end{equation}
for all Borel subsets $E\subseteq (0,\infty)$.  For $0 < \sigma < 2$ the measure $\mu_{\sigma}$ satisfies the 
condition (\ref{intro31}), and it satisfies (\ref{intro47}) if and only if $1 < \sigma < 2$. Observing that
\begin{align}\label{Intro27.1}
\begin{split}
f_{\mu_{\sigma}}(x) & = \int_0^{\infty} \big\{e^{-\lambda |x|} - e^{-\lambda}\big\}\ \lambda^{-\sigma}\ \dl \\
                    & = \Gamma(1 - \sigma)\big\{|x|^{\sigma - 1} - 1\big\}\, ,\quad\text{if $\sigma\not= 1$,}
\end{split}
\end{align}
one can apply Theorem \ref{thm1.3} and Theorem \ref{thm1.4} (in the case $1< \sigma<2$) to find the extremals of exponential type for the even function $x \mapsto |x|^{\sigma-1}$ where $0< \sigma < 2$ and $\sigma \neq 1$. We will return to these examples in section 7.

Our results can also be used to majorize and minorize certain real valued periodic functions by trigonometric 
polynomials.  This is accomplished by applying the Poisson summation formula to the functions that occur
in the inequality (\ref{intro6}), and then integrating the parameter $\lambda$ with respect to a measure $\mu$.
We give a general account of this method in section 6.
For example, if $\mu$ is the Haar measure defined by (\ref{intro80}), we obtain extremal trigonometric polynomials that 
majorize the periodic function $x\mapsto \log \bigl|1 - e(x)\bigr|$.  Here is the precise result.

\begin{theorem}\label{thm1.5}  Let $N$ be a nonnegative integer.  Then there exists a real valued trigonometric
polynomial
\begin{equation}\label{intro90}
u_N(x) = \sum_{n=-N}^N \tu_N(n) e(nx),
\end{equation}
such that
\begin{equation}\label{intro91}
\log \bigl|1 - e(x)\bigr| \le u_N(x)
\end{equation}
at each point $x$ in $\R/\Z$,
\begin{equation}\label{intro92}
\frac{\log 2}{N+1} = \int_{\R/\Z} u_N(x)\ \dx,
\end{equation}
and
\begin{equation}\label{intro93}
-\frac{1}{2|n|} \le \tu_N(n) \le 0
\end{equation}
for each integer $n$ with $1 \le |n| \le N$.  If $\widetilde{u}(x)$ is a real trigonometric polynomial of degree at 
most $N$ such that
\begin{equation*}\label{intro94}
\log \bigl|1 - e(x)\bigr| \le \widetilde{u}(x)
\end{equation*}
at each point $x$ in $\R/\Z$, then
\begin{equation}\label{intro95}
\frac{\log 2}{N+1} \le \int_{\R/\Z} \widetilde{u}(x)\ \dx.
\end{equation}
Moreover, there is equality in the inequality {\rm (\ref{intro95})} if and only if $\widetilde{u}(x) = u_N(x)$.
\end{theorem}

In section 8 we use (\ref{intro91}) to prove an analogue of the Erd\"os-Tur\'an inequality
for the supremum norm of an algebraic polynomial on the closed unit disk.


\section{Growth estimates in the complex plane}

Let $\rr = \{z\in\C: 0<\Re(z)\}$ denote the open right half plane.  Throughout this section we work with a function 
$\Phi(z)$ that is analytic on $\rr$ and satisfies the following conditions: If $0 < a < b < \infty$ then
\begin{equation}\label{cv10}
\lim_{y\rightarrow \pm\infty} e^{-2\pi |y|} \int_a^b \left|\frac{\Phi(x+iy)}{x+iy}\right|\ \dx = 0,
\end{equation}
if $0 < \eta < \infty$ then
\begin{equation}\label{cv11}
\sup_{\eta \le x} \int_{-\infty}^{\infty} \left|\frac{\Phi(x+iy)}{x+iy}\right| e^{-2\pi |y|}\ \dy < \infty,
\end{equation}
and
\begin{equation}\label{cv12}
\lim_{x\rightarrow \infty} \int_{-\infty}^{\infty} \left|\frac{\Phi(x+iy)}{x+iy}\right| e^{-2\pi |y|}\ \dy = 0.
\end{equation}

\begin{lemma}\label{lem2.1}  Assume that the analytic function $\Phi:\rr\rightarrow \C$ satisfies the 
conditions {\rm (\ref{cv10})}, {\rm (\ref{cv11})}, and {\rm (\ref{cv12})}, and let $0 < \delta$.  Then there 
exists a positive number $c(\delta,\Phi)$, depending only on $\delta$ and $\Phi$, such that the inequality
\begin{equation}\label{cv13}
|\Phi(z)| \le c(\delta,\Phi) |z| e^{2\pi |y|}
\end{equation}
holds for all $z = x+iy$ in the closed half plane $\{z\in\C: \delta\le \Re(z)\}$.
\end{lemma}

\begin{proof}  Write $\eta = \min\{\frac14, \h \delta\}$, and set
\begin{equation*}
c_1(\eta, \Phi) = \sup\left\{\int_{-\infty}^{\infty}
			\left|\frac{\Phi(u+iv)}{u+iv}\right|e^{-2\pi|v|}\ \dv : \eta \le u \right\}.
\end{equation*}
Then $c_1(\eta, \Phi)$ is finite by (\ref{cv11}).  Let $z = x+iy$ satisfy
$\delta \le \Re(z)$ and let $T$ be a positive real parameter such that $|y|+ \eta < T$.  Then write 
$\Gamma(z, \eta, T)$ for the simply connected, positively oriented, rectangular path connecting the
points $x-\eta - iT, x+\eta - iT, x+\eta + iT, x-\eta + iT,$ 
and $x-\eta - iT$.  From Cauchy's integral formula we have
\begin{equation}\label{cv14}
\frac{\Phi(z)}{z} = \frac{1}{2\pi i}\int_{\Gamma(z, \eta, T)}
			\frac{\Phi(w)}{w(w - z)\bigl(\cos \pi (w-z)\bigr)^2}\ \dw.
\end{equation}
At each point $w = u + iv$ on the path $\Gamma(z, \eta, T)$ we find that
\begin{equation}\label{cv15}
\eta \le |w - z|
\end{equation}
and 
\begin{align}\label{cv16}
\begin{split}
\frac{1}{|\cos \pi (w-z)|^2} &= \frac{2}{\bigl(\cos 2\pi (u-x) + \cosh 2\pi (v-y)\bigr)} \\
	&\le \frac{2}{\bigl(\cosh 2\pi (v-y)\bigr)} \\
	&\le 4e^{-2\pi |v-y|} \le 4e^{2\pi (|y| - |v|)}.
\end{split}
\end{align}
Using these estimates and (\ref{cv10}) we get
\begin{align}\label{cv17}
\begin{split}
\limsup_{T\rightarrow \infty}\ \Bigg|\int_{x-\eta\pm iT}^{x+\eta\pm iT}
                        &\frac{\Phi(w)}{w(w - z)\bigl(\cos \pi (w-z)\bigr)^2}\ \dw\Bigg| \\
	&\le \limsup_{T\rightarrow \infty} 4 \eta^{-1} e^{2\pi (|y| - T)} 
			 \int_{x-\eta}^{x+\eta}\left|\frac{\Phi(u\pm iT)}{u\pm iT}\right|\ \du \\
	&= 0.
\end{split}
\end{align}
It follows from (\ref{cv14}) and (\ref{cv17}) that
\begin{align}\label{cv18}
\begin{split}
\frac{\Phi(z)}{z} &= \frac{1}{2\pi i}\int_{x+\eta-i\infty}^{x+\eta+i\infty} 
				\frac{\Phi(w)}{w(w - z)\bigl(\cos \pi (w-z)\bigr)^2}\ \dw \\
	&\qquad\qquad -\frac{1}{2\pi i} \int_{x-\eta-i\infty}^{x-\eta+i\infty}
				\frac{\Phi(w)}{w(w - z)\bigl(\cos \pi (w-z)\bigr)^2}\ \dw.
\end{split}
\end{align}
By appealing to (\ref{cv15}) and (\ref{cv16}) again we find that
\begin{align}\label{cv19}
\begin{split}
\Bigg|\int_{x\pm \eta-i\infty}^{x\pm \eta+i\infty}
				&\frac{\Phi(w)}{w(w - z)\bigl(\cos \pi (w-z)\bigr)^2}\ \dw\Bigg|  \\
	&\le 4 \eta^{-1} e^{2\pi |y|}\int_{-\infty}^{\infty} 
			\left|\frac{\Phi(x\pm \eta +iv)}{x\pm \eta +iv}\right|e^{-2\pi |v|}\ \dv \\
	&\le 4 c_1(\eta, \Phi)\eta^{-1} e^{2\pi |y|}.
\end{split}
\end{align}
Combining (\ref{cv18}) and (\ref{cv19}) leads to the estimate
\begin{equation*}\label{cv20}
\left|\frac{\Phi(z)}{z}\right| \le 4(\pi \eta)^{-1} c_1(\eta, \Phi) e^{2\pi |y|},
\end{equation*}
and this plainly verifies (\ref{cv13}) with $c(\delta,\Phi) = 4(\pi \eta)^{-1}c_1(\eta,\Phi)$.
\end{proof}

Let $w=u+iv$ be a complex variable.  From (\ref{cv11}) we find that for each positive real number 
$\beta$ such that $\beta - \h$ is not an integer, and each complex number $z$ with $|\Re(z)|\not=\beta$, the function
\begin{equation*}\label{cv21}
w \mapsto \Bigl(\frac{\cos \pi z}{\cos \pi w}\Bigr)^2\Bigl(\frac{2w}{z^2 - w^2}\Bigr)\Phi(w)
\end{equation*}
is integrable along the vertical line $\Re(w) = \beta$.  We define a 
complex valued function $z\mapsto I(\beta, \Phi; z)$ on each component of the open set 
\begin{equation}\label{open}
\{z\in \C: \bigl|\Re(z)\bigr| \not= \beta\},
\end{equation}
by
\begin{equation}\label{cv22}
I(\beta, \Phi; z) = \frac{1}{2\pi i}\int_{\beta-i\infty}^{\beta+i\infty}\Bigl(\frac{\cos \pi z}{\cos \pi w}\Bigr)^2
				\Bigl(\frac{2w}{z^2 - w^2}\Bigr)\Phi(w)\ \dw.
\end{equation}
It follows using Morera's theorem that $z \mapsto I(\beta, \Phi; z)$ is analytic in each of the three components.  

In a similar manner we find that for each positive real number $\beta$ such that $\beta$ is not an integer, and each
complex number $z$ with $|\Re(z)|\not=\beta$, the function
\begin{equation*}\label{cv23}
w \mapsto \Bigl(\frac{\sin \pi z}{\sin \pi w}\Bigr)^2\Bigl(\frac{2w}{z^2 - w^2}\Bigr)\Phi(w)
\end{equation*}
is integrable along the vertical line $\Re(w) = \beta$.  We define a 
complex valued function $z\mapsto J(\beta, \Phi; z)$ on each component of the open set (\ref{open}) by
\begin{equation}\label{cv24}
J(\beta, \Phi; z) = \frac{1}{2\pi i}\int_{\beta-i\infty}^{\beta+i\infty}\Bigl(\frac{\sin \pi z}{\sin \pi w}\Bigr)^2
	\Bigl(\frac{2w}{z^2 - w^2}\Bigr)\Phi(w)\ \dw.
\end{equation}
Again Morera's theorem can be used to show that $J(\beta, \Phi; z)$ is analytic in each of the three components.

Next we prove a simple estimate for $I(\beta, \Phi; z)$ and $J(\beta, \Phi; z)$.

\begin{lemma}\label{lem2.2}
Assume that the analytic function $\Phi:\rr \rightarrow \C$ satisfies the conditions
{\rm (\ref{cv10})}, {\rm (\ref{cv11})}, and {\rm (\ref{cv12})}.  Let $\beta$ be a positive real number,
$z = x + iy$ a complex number such that $|\Re(z)|\not=\beta$, and write
\begin{equation}\label{cv25}
B(\beta, \Phi) = \frac{4}{\pi} 
	\int_{-\infty}^{+\infty} \left|\frac{\Phi(\beta + iv)}{\beta + iv}\right|e^{-2\pi|v|}\ \dv.
\end{equation}
If $\beta - \h$ is not an integer then
\begin{equation}\label{cv26}
|I(\beta, \Phi; z)| \le B(\beta, \Phi) \sec^2 \pi\beta 
	\left(1 + \frac{|z|}{\bigl||x|-\beta\bigr|}\right)e^{2\pi |y|}.
\end{equation}
If $\beta$ is not an integer then
\begin{equation}\label{cv27}
|J(\beta, \Phi; z)| \le B(\beta, \Phi) \csc^2 \pi\beta
	\left(1 + \frac{|z|}{\bigl||x|-\beta\bigr|}\right)e^{2\pi |y|}.
\end{equation}
\end{lemma}

\begin{proof}  
On the vertical line $\Re(w)=\beta$ we have
\begin{equation*}\label{cv28}
\bigl||x|-\beta \bigr|\le \min\{|z-w|,|z+w|\}
\end{equation*}
and
\begin{equation*}\label{cv29}
|z|\le \hh|z-w|+\hh|z+w|\le \max\{|z-w|,|z+w|\},
\end{equation*}
and therefore
\begin{align}\label{cv30}
\begin{split}
\Bigl|\frac{w^2}{z^2-w^2}\Bigr| &\le 1 + \Bigl|\frac{z^2}{z^2-w^2}\Bigr| \\
	&= 1 + |z|^2\big(\min\{|z-w|,|z+w|\}\max\{|z-w|,|z+w|\}\bigr)^{-1} \\
	&\le 1 + \frac{|z|}{\bigl||x|-\beta\bigr|}.
\end{split}
\end{align}
On the line $\Re(w)=\beta$ we also use the elementary inequality
\begin{equation}\label{cv31}
|\cos \pi (\beta + iv)|^{-2} \le 4 e^{-2\pi |v|} \sec^2 \pi\beta.
\end{equation}
Then we use (\ref{cv30}) and (\ref{cv31}) to estimate the integral on the right of (\ref{cv22}).  The bound
(\ref{cv26}) follows easily.

The proof of (\ref{cv27}) is very similar.
\end{proof}

For each positive number $\xi$ we define an even rational function $z\mapsto \A(\xi, \Phi; z)$ 
on $\C$ by
\begin{align}\label{cv32}
\begin{split}
\A(\xi, \Phi; z) &= \Phi(\xi)(z - \xi)^{-2} + \Phi^{\prime}(\xi)(z - \xi)^{-1} \\ 
	  &\qquad + \Phi(\xi)(z + \xi)^{-2} - \Phi^{\prime}(\xi)(z + \xi)^{-1}.
\end{split}
\end{align}

\begin{lemma}\label{lem2.3}  
Assume that the analytic function $\Phi:\rr \rightarrow \C$ satisfies the conditions
{\rm (\ref{cv10})}, {\rm (\ref{cv11})}, and {\rm (\ref{cv12})}.  Then the sequence of entire functions
\begin{equation}\label{cv33}
\Bigl(\frac{\cos \pi z}{\pi}\Bigr)^2\sum_{n=1}^N\A(n - \hh, \Phi; z),\ \text{where}\ N = 1, 2, 3, \dots ,
\end{equation}
converges uniformly on compact subsets of $\C$ as $N\rightarrow \infty$, and therefore
\begin{equation}\label{cv34}
\G(\Phi, z) = \lim_{N\rightarrow \infty}\Bigl(\frac{\cos \pi z}{\pi}\Bigr)^2 \sum_{n=1}^N \A(n - \hh, \Phi; z)
\end{equation}
defines an entire function.  
Also, the sequence of entire functions
\begin{equation}\label{cv35}
\Bigl(\frac{\sin \pi z}{\pi}\Bigr)^2\sum_{n=1}^{N}\A(n, \Phi; z),\ \text{where}\ N = 1, 2, 3, \dots ,
\end{equation}
converges uniformly on compact subsets of $\C$ as $N\rightarrow \infty$, and therefore
\begin{equation}\label{cv36}
\sH(\Phi, z) = \lim_{N\rightarrow \infty}\Bigl(\frac{\sin \pi z}{\pi}\Bigr)^2 \sum_{n=1}^{N} \A(n, \Phi; z)
\end{equation}
defines an entire function.  
\end{lemma}

\begin{proof} We assume that $z$ is a complex number in $\rr$ such that $z-\h$ is not an integer.  Then
\begin{equation}\label{cv37}
w\mapsto \Bigl(\frac{\cos \pi z}{\cos \pi w}\Bigr)^2\Bigl(\frac{2w}{z^2 - w^2}\Bigr)\Phi(w)
\end{equation}
defines a meromorphic function of $w$ on the right half plane $\rr$.  We find that (\ref{cv37}) has a
simple pole at $w = z$ with residue $-\Phi(z)$.  And for each positive integer $n$, (\ref{cv37}) has a 
pole of order at most two at $w = n - \h$ with residue
\begin{equation*}
\Bigl(\frac{\cos \pi z}{\pi}\Bigr)^2 \A(n - \hh, \Phi; z). 
\end{equation*} 
Plainly (\ref{cv37}) has no other poles in $\rr$.  Let $0 < \beta < \h$, let $N$ be a positive integer, and $T$ a
positive real parameter.  Write $\Gamma(\beta, N, T)$ for the simply connected, positively oriented
rectangular path connecting the points $\beta - iT$, $N - iT$, $N + iT$, $\beta + iT$ and $\beta - iT$.
If $z$ satisfies $\beta < \Re(z) < N$ and $|\Im(z)| < T$, and $z-\h$ is not an 
integer, then from the residue theorem we obtain the identity
\begin{align}\label{cv38}
\begin{split}
\Bigl(\frac{\cos \pi z}{\pi}\Bigr)^2 &\sum_{n=1}^N \A(n - \hh, \Phi; z) - \Phi(z) \\
	&=\frac{1}{2\pi i}\int_{\Gamma(\beta, N, T)} \Bigl(\frac{\cos \pi z}{\cos \pi w}
		\Bigr)^2\Bigl(\frac{2w}{z^2 - w^2}\Bigr)\Phi(w)\ \dw.
\end{split}
\end{align}
We let $T\rightarrow \infty$ on the right hand side of (\ref{cv38}), and we use the hypotheses (\ref{cv10})
and (\ref{cv11}).  In this way we conclude that
\begin{equation}\label{cv39}
\Bigl(\frac{\cos \pi z}{\pi}\Bigr)^2 \sum_{n=1}^N \A(n - \hh, \Phi; z) - \Phi(z) = I(N, \Phi; z) - I(\beta, \Phi; z).
\end{equation}
Initially (\ref{cv39}) holds for $\beta < \Re(z) < N$ and $z-\h$ not an integer.  However, we have already observed
that both sides of (\ref{cv39}) are analytic in the strip $\{z\in \C: \beta < \Re(z) < N\}$.  Therefore the condition
that $z-\h$ is not an integer can be dropped.

Now let $M < N$ be positive integers.  From (\ref{cv39}) we find that
\begin{equation}\label{cv40}
\Bigl(\frac{\cos \pi z}{\pi}\Bigr)^2 \sum_{n=M+1}^N \A(n - \hh, \Phi; z)= I(N, \Phi; z) - I(M, \Phi; z)
\end{equation}
in the infinite strip $\{z\in \C: \beta < \Re(z) < M\}$.  In fact we have seen that both sides of (\ref{cv40})
are analytic in the infinite strip $\{z\in \C: |\Re(z)| < M\}$.  Therefore the identity (\ref{cv40}) must hold in
this larger domain by analytic continuation.  Let $\K\subseteq \C$ be a compact set and assume that $L$ is an 
integer so large that $\K \subseteq \{z\in \C: 2|z| < L\}$.  From (\ref{cv12}), Lemma \ref{lem2.2}, and 
(\ref{cv40}), it is obvious that the sequence of entire functions (\ref{cv33}), where $L \le N$,
is uniformly Cauchy on $\K$.  This verifies the first assertion of the lemma and shows that (\ref{cv34}) 
defines an entire function.  The second assertion of the lemma can be established in essentially the same manner.
\end{proof}

\begin{lemma}\label{lem2.4}  Assume that the analytic function $\Phi:\rr \rightarrow \C$ satisfies the conditions
{\rm (\ref{cv10})}, {\rm (\ref{cv11})} and {\rm (\ref{cv12})}.  Let the entire functions $\G(\Phi, z)$ and $\sH(\Phi, z)$ be
defined by {\rm (\ref{cv34})} and {\rm (\ref{cv36})}, respectively.  If $0 < \beta < \h$ then the identity 
\begin{equation}\label{cv41}
\Phi(z) - \G(\Phi, z) =  I(\beta, \Phi; z)
\end{equation} 
holds for all $z$ in the half plane $\{z\in \C: \beta < \Re(z)\}$, and the identity
\begin{equation}\label{cv42}
-\G(\Phi, z) =  I(\beta, \Phi; z)
\end{equation}
holds for all $z$ in the infinite strip $\{z\in \C: |\Re(z)| < \beta\}$.  If $0 < \beta < 1$ then the identity
\begin{equation}\label{cv43}
\Phi(z) - \sH(\Phi, z) =  J(\beta, \Phi; z)
\end{equation} 
holds for all $z$ in the half plane $\{z\in \C: \beta < \Re(z)\}$, and the identity
\begin{equation}\label{cv44}
-\sH(\Phi, z) =  J(\beta, \Phi; z)
\end{equation}
holds for all $z$ in the infinite strip $\{z\in \C: |\Re(z)| < \beta\}$.
\end{lemma}

\begin{proof}  We argue as in the proof of Lemma \ref{lem2.3}, letting $N\rightarrow \infty$ on both sides of
(\ref{cv39}).  Then we use (\ref{cv12}) and Lemma \ref{lem2.2}, and obtain the identity
\begin{equation*}\label{cv45}
\Phi(z) - \G(\Phi, z) = I(\beta, \Phi; z)
\end{equation*}
at each point of the half plane $\{z\in \C: \beta < \Re(z)\}$.  This proves (\ref{cv41}).

Next, we assume that $|\Re(z)| < \beta$.  In this case the residue theorem provides the identity
\begin{align}\label{cv46}
\begin{split}
\Bigl(\frac{\cos \pi z}{\pi}\Bigr)^2 &\sum_{n=1}^N \A(n - \hh, \Phi; z) \\
	&= \frac{1}{2\pi i}\int_{\Gamma(\beta, N, T)} \Bigl(\frac{\cos \pi z}{\cos \pi w}
		\Bigr)^2\Bigl(\frac{2w}{z^2 - w^2}\Bigr)\Phi(w)\ \dw.
\end{split}
\end{align}
We let $T\rightarrow \infty$ and argue as before.  In this way (\ref{cv46}) leads to 
\begin{equation}\label{cv47}
\Bigl(\frac{\cos \pi z}{\pi}\Bigr)^2 \sum_{n=1}^N \A(n - \hh, \Phi; z) = I(N, \Phi; z) - I(\beta, \Phi; z).
\end{equation}
Then we let $N\rightarrow \infty$ on both sides of (\ref{cv47}) and we use (\ref{cv12}) and Lemma \ref{lem2.2} 
again.  We find that
\begin{equation*}
 -\G(\Phi, z) = I(\beta, \Phi; z),
\end{equation*}
and this verifies (\ref{cv42}).

The identities (\ref{cv43}) and (\ref{cv44}) are obtained in the same way.
\end{proof}

\begin{corollary}\label{cor2.5} 
Suppose that $\Phi(z) = 1$ is constant on $\rr$.  If $0 < \beta < \h$ then 
\begin{equation}\label{cv48}
I(\beta, 1; z) = 0,
\end{equation}
in the open half plane $\{z\in\C: \beta < \Re(z)\}$.  If $0 < \beta < 1$ then
\begin{equation}\label{cv49}
J(\beta, 1; z) = \left(\frac{\sin \pi z}{\pi z}\right)^2,
\end{equation}
in the open half plane $\{z\in\C: \beta < \Re(z)\}$.
\end{corollary}

\begin{proof}  We have
\begin{align*}
\begin{split}
\G(1, z) &= \lim_{N\rightarrow \infty}\left(\frac{\cos \pi z}{\pi}\right)^2 \sum_{n=1}^N \A(n - \hh, 1; z) \\
     &= \lim_{N\rightarrow \infty}\left(\frac{\cos \pi z}{\pi}\right)^2 \sum_{n=-N}^{N-1} (z-n-\hh)^{-2} = 1.
\end{split}
\end{align*}
Now the identity (\ref{cv48}) follows from (\ref{cv41}).  In a similar manner,
\begin{align*}
\sH(1, z) &= \lim_{N\rightarrow \infty}\left(\frac{\sin \pi z}{\pi}\right)^2 \sum_{n=1}^N \A(n, 1; z) \\
     &= \lim_{N\rightarrow \infty} \left(\frac{\sin \pi z}{\pi}\right)^2 \sum_{n=-N}^N (z - n)^{-2}
     		- \left(\frac{\sin \pi z}{\pi z}\right)^2 \\
     &= 1 - \left(\frac{\sin \pi z}{\pi z}\right)^2,
\end{align*}
and (\ref{cv49}) follows from (\ref{cv43}).
\end{proof}

\begin{lemma}\label{lem2.6}
Assume that the analytic function $\Phi:\rr \rightarrow \C$ satisfies the conditions
{\rm (\ref{cv10})}, {\rm (\ref{cv11})} and {\rm (\ref{cv12})}.  Let the entire functions $\G(\Phi, z)$ and $\sH(\Phi, z)$ 
be defined by {\rm (\ref{cv34})} and {\rm (\ref{cv36})}, respectively.  Then there exists a positive number 
$c(\Phi)$, depending only on $\Phi$, such that the inequalities
\begin{equation}\label{cv51}
|\G(\Phi, z)|\le c(\Phi)(1 +|z|)e^{2\pi |y|},
\end{equation}
and
\begin{equation}\label{cv52}
|\sH(\Phi, z)| \le c(\Phi)(1 +|z|)e^{2\pi |y|},
\end{equation}
hold for all complex numbers $z=x+iy$.  In particular, both $\G(\Phi, z)$ and $\sH(\Phi, z)$ are entire functions
of exponential type at most $2\pi$.
\end{lemma}

\begin{proof}
In the closed half plane $\{z\in \C: \frac14 \le \Re(z)\}$ the identity (\ref{cv41}) implies that
\begin{equation*}
|\G(\Phi, z)| \le |\Phi(z)| + |I(\tfrac{1}{8}, \Phi; z)|.
\end{equation*}
Then an estimate of the form (\ref{cv51}) in this half plane follows from Lemma \ref{lem2.1} and Lemma \ref{lem2.2}.
In the closed infinite strip $\{z\in \C: |\Re(z)| \le \frac14\}$ we have
\begin{equation*}
|\G(\Phi, z)| = |I(\tfrac{3}{8}, \Phi; z)|
\end{equation*}
from the identity (\ref{cv42}).  Plainly an estimate of the form (\ref{cv51}) in this closed infinite
strip follows from Lemma \ref{lem2.2}.  This proves the inequality (\ref{cv51}) for all complex 
$z$ because $\G(\Phi, z)$ is an even function of $z$.  The inequality (\ref{cv52}) is established
in the same manner using $J(\beta, \Phi; z)$ in place of $I(\beta, \Phi; z)$.
\end{proof}


\section{Fourier expansions}

It follows directly from the definition (\ref{intro4}) that $z\mapsto L(\lambda,z)$ interpolates the 
values of the function $x\mapsto e^{-\lambda |x|}$ and its derivative at points of the coset $\Z + \h$.  
That is, the identities
\begin{equation}\label{gv3}
L(\lambda, k+\hh) = e^{-\lambda |k+\h|}\quad\text{and}\quad 
				L^{\prime}(\lambda, k+\hh) = -\sgn(k+\hh)\lambda e^{-\lambda |k+\h|}
\end{equation}
hold for each integer $k$.  Similarly, it follows from (\ref{intro5}) that $z\mapsto M(\lambda,z)$ 
interpolates the values of the function $x\mapsto e^{-\lambda |x|}$ at points of $\Z$ and interpolates 
its derivative at points of $\Z\setminus\{0\}$.  Thus we get
\begin{equation}\label{gv4}
M(\lambda,l) = e^{-\lambda |l|}\quad\text{and}\quad M^{\prime}(\lambda,l) = -\sgn(l)e^{-\lambda |l|}
\end{equation}
for each integer $l$.

\begin{lemma}\label{lem3.2}  If $0 < \beta < \h$, then at each point $z$ in the half plane 
$\{z\in\C:\beta < \Re(z)\}$ we have
\begin{equation}\label{cv-1}
e^{-\lambda z}- L(\lambda, z) = \dfrac{1}{2\pi i} \int_{\beta -i \infty}^{\beta + i \infty} 
	\left(\dfrac{\cos \pi z}{\cos \pi w}\right)^2 \left(\dfrac{2w}{z^2 - w^2}\right) e^{-w\lambda}\ \dw.
\end{equation}
If $0 < \beta < 1$, then at each point $z$ in the half plane $\{z\in\C:\beta < \Re(z)\}$ we have
\begin{equation}\label{cv0}
M(\lambda, z)- e^{-\lambda z} = \dfrac{1}{2\pi i} \int_{\beta -i \infty}^{\beta + i \infty} 
	\left(\dfrac{\sin \pi z}{\sin \pi w}\right)^2 \left(\dfrac{2w}{z^2 - w^2}\right) 
		\left(1 - e^{-w\lambda}\right)\ \dw. 
\end{equation}
\end{lemma}

\begin{proof}  We apply Lemma \ref{lem2.3} with $\Phi(z) = e^{-z\lambda}$.  It follows that
\begin{equation*}
\G(\Phi, z) = L(\lambda, z)\quad\text{and}\quad 
		\sH(\Phi, z) = M(\lambda, z) - \left(\frac{\sin \pi z}{\pi z}\right)^2.
\end{equation*}
The identities (\ref{cv-1}) and (\ref{cv0}) follow now from Lemma \ref{lem2.4} and Corollary
\ref{cor2.5}.
\end{proof}

As $x\mapsto L(\lambda,x)$ and $x\mapsto M(\lambda,x)$ are both bounded and 
integrable on $\R$, their Fourier transforms
\begin{equation}\label{finite3}
\tL(\lambda,t) = \int_{-\infty}^{\infty}L(\lambda,x)e(-tx)\ \dx\quad\text{and}\quad
	\tM(\lambda,t) = \int_{-\infty}^{\infty}M(\lambda,x)e(-tx)\ \dx
\end{equation}
are continuous functions of the real variable $t$ supported on the interval $[-1,1]$.  Then 
by Fourier inversion we have the representations
\begin{equation}\label{finite4}
L(\lambda,z) = \int_{-1}^{1}\tL(\lambda,t)e(tz)\ \dt\quad\text{and}\quad 
		M(\lambda,z) = \int_{-1}^{1}\tM(\lambda,t)e(tz)\ \dt
\end{equation}
for all complex $z$.  It will be useful to have more explicit information about the Fourier transforms 
of these functions.

\begin{lemma}\label{lem3.3}  For $|t| \le 1$ the Fourier transforms {\rm (\ref{finite3})} are given by
\begin{equation}\label{finite5}
\tL(\lambda, t) = \dfrac{(1-|t|)\sinh\left(\frac{\lambda}{2}\right) \cos \pi t + \frac{\lambda}{2\pi}
	|\sin \pi t|\cosh\left(\frac{\lambda}{2}\right)}{\sinh^2\left(\frac{\lambda}{2}\right) + \sin ^2\pi t}, 
\end{equation}
and
\begin{equation}\label{finite6}
\tM(\lambda, t) = \dfrac{(1-|t|)\sinh\left(\frac{\lambda}{2}\right)\cosh\left(\frac{\lambda}{2}\right) 
	+ \frac{\lambda}{2\pi}|\sin \pi t|\cos \pi t}{\sinh^2\left(\frac{\lambda}{2}\right) + \sin ^2\pi t}.
\end{equation}
Moreover, we have
\begin{equation}\label{pft1}
0 \le \tL(\lambda, t)\quad\text{and}\quad 0 \le \tM(\lambda, t)
\end{equation}
for all real $t$.
\end{lemma}

\begin{proof}  The Fourier transform $\tL(\lambda,t)$ can be explicitly determined as follows. For 
$\lambda > 0$ we define, as in \cite[equation (3.1)]{GV}, the entire function
\begin{equation*}
A(\lambda, z)  =  \left(\dfrac{\sin \pi z}{\pi}\right)^2 
	\sum_{n=0}^{\infty} e^{-\lambda n} \left\{ (z-n)^{-2} - \lambda(z-n)^{-1}\right\}.
\end{equation*}
Then $z\mapsto A(\lambda,z)$ has exponential type $2\pi$ and its restriction to $\R$ is in $L^2(\R)$. Using 
\cite[Theorem 9]{V} we find that
\begin{equation*}
A(\lambda, z) = \int_{-1}^1\widehat{A}(\lambda,t)e(tz)\ \dt
\end{equation*}
for all complex $z$, where
\begin{equation}\label{finite7}
\widehat{A}(\lambda,t) = (1-|t|)u_{\lambda}(t) + (2\pi i)^{-1} \sgn(t) v_{\lambda}(t)	
\end{equation}
with
\begin{equation*}
u_{\lambda}(t) = \sum_{m=0}^{\infty}e^{-\lambda m - 2\pi i mt} =  \left(1 - e^{-\lambda - 2\pi i t}\right)^{-1},
\end{equation*}
and
\begin{equation*}
v_{\lambda}(t) = -\lambda \sum_{m=0}^{\infty}e^{-\lambda m - 2\pi i mt}
			= -\lambda \left(1 - e^{-\lambda - 2\pi i t}\right)^{-1}.
\end{equation*}
Therefore (\ref{finite7}) can be written as
\begin{equation*}
\widehat{A}(\lambda,t) = \left\{\bigl(1-|t|\bigr) - \frac{\lambda}{2\pi i}\sgn(t)\right\}
	\left(1 - e^{-\lambda - 2\pi i t}\right)^{-1}	
\end{equation*}
for $|t| \leq 1$. Next we observe that
\begin{align*}
L(\lambda, z) &= e^{-\frac{\lambda}{2}} \left\{ A\left(\lambda, z - \hh\right) + A\left(\lambda, -z -\hh\right) \right\}\\
	      &= e^{-\frac{\lambda}{2}} \left\{ \int_{-1}^1\widehat{A}(\lambda,t)e\left(t(z - \hh)\right)\ \dt + 						
	      \int_{-1}^1\widehat{A}(\lambda,-t)e\left(t(z + \hh)\right)\ \dt\right\}.
\end{align*}
It follows that
\begin{align}\label{finite8}\begin{split}
\tL(\lambda, t) &= e^{-\frac{\lambda}{2}} \left\{\widehat{A}(\lambda,t)e\left(- \hh t\right)
	+ \widehat{A}(\lambda,-t)e\left( \hh t\right)\right\}\\
			& = \dfrac{(1-|t|)\sinh\left(\frac{\lambda}{2}\right) \cos \pi t + \frac{\lambda}{2\pi}
	|\sin \pi t|\cosh\left(\frac{\lambda}{2}\right)}{\sinh^2\left(\frac{\lambda}{2}\right) + \sin ^2\pi t}
\end{split}\end{align}
for $|t| \leq 1$.  In a similar manner we use
\begin{equation*}
M(\lambda,z) = A(\lambda,z) + A(\lambda, -z) - \left(\dfrac{\sin \pi z}{\pi z}\right)^2,
\end{equation*}
and the identity
\begin{equation*}
\left(\dfrac{\sin \pi z}{\pi z}\right)^2 = \int_{-1}^{1}(1 - |t|)e(tz)\ \dt.
\end{equation*}
We find that
\begin{equation}\label{finite9}
\tM(\lambda, t) = \dfrac{(1-|t|)\sinh\left(\frac{\lambda}{2}\right)\cosh\left(\frac{\lambda}{2}\right) 
	+ \frac{\lambda}{2\pi}|\sin \pi t|\cos \pi t}{\sinh^2\left(\frac{\lambda}{2}\right) + \sin ^2\pi t}. 
\end{equation}
It follows now from (\ref{finite8}) and (\ref{finite9}) that both $\tL(\lambda,t)$ 
and $\tM(\lambda,t)$ are nonnegative for all real $t$.
\end{proof}

For later applications is will be useful to have the following inequality.

\begin{corollary}\label{cor3.1}  If $0 < |t| \le 1$ then we have
\begin{equation}\label{ineq1}
\int_0^{\infty} \tL(\lambda, t)\lambda^{-1}\ \dl \le \frac{1}{2|t|}.
\end{equation}
\end{corollary}

\begin{proof} For $0 < |t| \le 1$ we use the elementary inequalities
\begin{equation*}\label{ineq2}
\cos \pi t \le \frac{\sin \pi t}{\pi t},\quad\text{and}
	\quad \sinh\left(\frac{\lambda}{2}\right) \le \frac{\lambda}{2}\cosh\left(\frac{\lambda}{2}\right).
\end{equation*}
Then it follows from (\ref{finite5}) that
\begin{equation*}\label{ineq3}
\tL(\lambda, t) \le \Bigl(\frac{\sin \pi t}{\pi t}\Bigr)\dfrac{\frac{\lambda}{2}
	\cosh\left(\frac{\lambda}{2}\right)}{\sinh^2\left(\frac{\lambda}{2}\right) + \sin ^2\pi t},
\end{equation*}
and
\begin{equation*}\label{ineq4}
\int_0^{\infty} \tL(\lambda, t)\lambda^{-1}\ \dl 
	\le \Bigl(\frac{\sin \pi t}{\pi t}\Bigr)\int_0^{\infty} \dfrac{\frac{1}{2}
	\cosh\left(\frac{\lambda}{2}\right)}{\sinh^2\left(\frac{\lambda}{2}\right) + \sin ^2\pi t}\ \dl \\
	= \frac{1}{2|t|}. 
\end{equation*}  
\end{proof}
\begin{remark}\label{rem3.5}
In fact, Corollary \ref{cor3.1} is a particular application of the following more general upper bound 
\begin{equation}\label{Sec3}
 \widehat{L}(\lambda,t) \leq \dfrac{2\lambda}{\lambda^2 + 4 \pi^2 t^2}
\end{equation}
for all $\lambda >0$ and $t \in \R$. This bound may be useful in other applications. One can prove 
(\ref{Sec3}) by clearing denominators, expanding in Taylor series with respect to $\lambda$ and observing 
that all coefficients (which are now functions of $t$ only) are nonnegative.
\end{remark}

\begin{lemma}\label{lem3.4}  Let $\nu$ be a finite measure on the Borel subsets of $(0,\infty)$.
For each complex number $z$ the functions $\lambda\mapsto L(\lambda,z)$ and $\lambda\mapsto M(\lambda,z)$ are 
$\nu$-integrable on $(0,\infty)$.  The complex valued functions
\begin{equation}\label{finite10}
L_{\nu}(z) =\int_{0}^{\infty} L(\lambda,z)\ \dnu\quad\text{and}\quad 
		M_{\nu}(z) =\int_{0}^{\infty} M(\lambda,z)\ \dnu
\end{equation}
are entire functions which satisfy the inequalities
\begin{equation}\label{finite11}
|L_{\nu}(z)| \leq \nu\{(0,\infty)\}e^{2\pi|y|}\quad\text{and}\quad |M_{\nu}(z)| \leq \nu\{(0,\infty)\}e^{2\pi|y|}
\end{equation}
for all $z = x + iy$.  In particular, both $L_{\nu}(z)$ and $M_{\nu}(z)$ are entire functions of exponential 
type at most $2\pi$.
\end{lemma}

\begin{proof}  We apply (\ref{finite4}) and the fact that $0 \le \tL(\lambda,t)$.  We find that
\begin{align}\label{finite12}
\begin{split}
\int_{0}^{\infty} \left|L(\lambda,z)\right|\ \dnu &=  \int_{0}^{\infty} 
		\left|\int_{-1}^1\tL(\lambda,t)e(tz)\ \dt\right|\ \dnu \\
	&\leq \int_{0}^{\infty} \int_{-1}^1\tL(\lambda,t) e^{-2\pi ty}\ \dt\ \dnu \\
	&\leq e^{2\pi |y|} \int_{0}^{\infty} \int_{-1}^1\tL(\lambda,t)\ \dt\ \dnu \\ 
 	&=  e^{2\pi |y|}\int_{0}^{\infty} L(\lambda,0)\ \dnu.
\end{split}
\end{align}
As $L(\lambda,0) \le 1$ by (\ref{intro6}), it follows from (\ref{finite12}) that
\begin{equation*}
\int_{0}^{\infty} \left|L(\lambda,z)\right|\ \dnu \leq \nu\{(0,\infty)\} e^{2\pi |y|}.
\end{equation*}
This shows that $\lambda\mapsto L(\lambda,z)$
is $\nu$-integrable on $(0,\infty)$ and verifies the bound on the left of (\ref{finite11}). 

In a similar manner we get
\begin{equation}\label{finite13}
\int_{0}^{\infty} \left|M(\lambda,z)\right|\ \dnu \le e^{2\pi |y|}\int_{0}^{\infty} M(\lambda,0)\ \dnu.
\end{equation}
It is clear from (\ref{gv4}) that $z\mapsto M(\lambda,z)$ interpolates the values of the function 
$x\mapsto e^{-\lambda |x|}$ at the integers.  In particular, $M(\lambda,0) = 1$, and therefore (\ref{finite13})
implies that
\begin{equation*}
\int_{0}^{\infty} \left|M(\lambda,z)\right|\ \dnu \le \nu\{(0,\infty)\} e^{2\pi |y|}.
\end{equation*}
Again this shows that $\lambda\mapsto M(\lambda,z)$ is $\nu$-integrable and verifies the bound on the right
of (\ref{finite11}).

It follows easily using Morera's theorem that both $z\mapsto L_{\nu}(z)$ and $z\mapsto M_{\nu}(z)$ are
entire functions.  Then (\ref{finite11}) implies that both of these entire functions have exponential type at 
most $2\pi$.
\end{proof}

Let $\nu$ be a finite measure on the Borel subsets of $(0,\infty)$.  It follows that
\begin{equation}\label{def0}
\Psi_{\nu}(z) = \int_0^{\infty} e^{-\lambda z}\ \dnu
\end{equation}
defines a function that is bounded and continuous in the closed half plane $\{z\in\C: 0\le \Re(z)\}$, and analytic 
in the interior of this half plane.  

\begin{lemma}\label{lem3.5}  If $0 < \beta < \h$, then at each point $z$ in the half plane 
$\{z\in\C:\beta < \Re(z)\}$ we have
\begin{equation}\label{cv6}
\Psi_{\nu}(z) - L_{\nu}(z) = \dfrac{1}{2\pi i} \int_{\beta - i\infty}^{\beta + i\infty} 
	\left(\dfrac{\cos \pi z}{\cos \pi w}\right)^2\left(\dfrac{2w}{z^2 - w^2}\right)\Psi_{\nu}(w)\ \dw.
\end{equation}
If $0 < \beta < 1$ and $a_{\nu} = \nu\{(0,\infty)\}$, then at each point $z$ in the half plane
$\{z\in\C:\beta < \Re(z)\}$ we have 
\begin{equation}\label{cv7}
M_{\nu}(z) - \Psi_{\nu}(z) = \dfrac{1}{2\pi i} 
	\int_{\beta - i \infty}^{\beta + i \infty} \left(\dfrac{\sin \pi z}{\sin \pi w}\right)^2 
                \left( \dfrac{2w}{z^2 - w^2} \right) \bigl(a_{\nu} - \Psi_{\nu}(w)\bigr)\ \dw.
\end{equation}
\end{lemma}

\begin{proof}  We apply (\ref{cv-1}) and get
\begin{align*}
\Psi_{\nu}(z) - &L_{\nu}(z) \\
	     &= \int_{0}^{\infty} \left\{e^{-\lambda z} - L(\lambda, z)\right\}\ \dnu  \\
	     &= \int_{0}^{\infty} \left\{\dfrac{1}{2\pi i} 
	     	\int_{\beta -i \infty}^{\beta + i \infty} \left( \dfrac{\cos \pi z}{\cos \pi w}\right)^2 
		\left( \dfrac{2w}{z^2 - w^2} \right) e^{-w\lambda}\ \dw \right\}\ \dnu \\
	     &= \dfrac{1}{2\pi i} \int_{\beta -i \infty}^{\beta + i \infty} 
		\left( \dfrac{\cos \pi z}{\cos \pi w}\right)^2 \left( \dfrac{2w}{z^2 - w^2} \right) \Psi_{\nu}(w)\ \dw.
\end{align*}
This proves (\ref{cv6}).  Then (\ref{cv0}) leads to (\ref{cv7}) in the same manner.
\end{proof}


\section{Proof of Theorem \ref{thm1.1}}

Let $\mu$ be a measure defined on the Borel subsets of $(0,\infty)$ that satisfies (\ref{intro31}).
Let $z = x+iy$ be a point in the open right half plane $\rr = \{z\in\C: 0<\Re(z)\}$.  Using (\ref{intro31}) we 
find that
\begin{equation*}\label{pt-1}
\lambda\mapsto e^{-\lambda z} - e^{-\lambda}
\end{equation*}
is integrable on $(0,\infty)$ with respect to $\mu$.  We define $F_{\mu}:\rr\rightarrow \C$ by
\begin{equation}\label{pt0}
F_{\mu}(z) = \int_0^{\infty} \big\{e^{-\lambda z} - e^{-\lambda}\big\}\ \dmu.
\end{equation}
It follows by applying Morera's theorem that $F_{\mu}(z)$ is analytic on $\rr$.  Also, at each point $z$ 
in $\rr$ the derivative of $F_{\mu}$ is given by
\begin{equation}\label{pt1}
F_{\mu}^{\prime}(z) = - \int_0^{\infty} \lambda e^{-\lambda z}\ \dmu.
\end{equation}
Then (\ref{pt1}) leads to the bound
\begin{equation}\label{pt2}
\bigl|F_{\mu}^{\prime}(x + iy)\bigr| \le \int_0^{\infty} \lambda e^{-\lambda x}\ \dmu = \bigl|F_{\mu}^{\prime}(x)\bigr|.
\end{equation}
Using (\ref{pt2}) and the dominated convergence theorem we conclude that
\begin{equation}\label{pt3}
\lim_{x\rightarrow \infty} \bigl|F_{\mu}^{\prime}(x + iy)\bigr| = 0
\end{equation}
uniformly in $y$.  Clearly the functions $f_{\mu}(x)$, defined by (\ref{intro33}), and $F_{\mu}(z)$, defined 
by (\ref{pt0}), satisfy the identities
\begin{equation}\label{pt4}
f_{\mu}(x) = F_{\mu}\bigl(|x|\bigr)\quad\text{and}\quad f_{\mu}^{\prime}(x) = \sgn(x)F_{\mu}^{\prime}\bigl(|x|\bigr)
\end{equation}
for all real $x \not= 0$.

\begin{lemma}\label{lem4.1}  The analytic function $F_{\mu}(z)$ defined by {\rm (\ref{pt0})} satisfies each of the
three conditions {\rm (\ref{cv10})}, {\rm (\ref{cv11})}, and {\rm (\ref{cv12})}.
\end{lemma}

\begin{proof}  Let $0 < \xi \le 1$.  If $\xi \le \Re(z)$, then from (\ref{pt2}) we obtain the inequality
\begin{align*}\label{pt5}
\begin{split}
\bigl|F_{\mu}(z)\bigr| &= \left| \int_1^z F_{\mu}^{\prime}(w)\ \dw\right| \\
	&\le |z - 1| \max\big\{\bigl|F_{\mu}^{\prime}(\theta z + 1 - \theta)\bigr|: 0 \le \theta \le 1\big\} \\
	&\le (|z| + 1) \bigl|F_{\mu}^{\prime}(\xi)\bigr|,
\end{split}
\end{align*}
and therefore
\begin{equation}\label{pt6}
\left|\frac{F_{\mu}(z)}{z}\right|\le (1 + \xi^{-1})\bigl|F_{\mu}^{\prime}(\xi)\bigr|.
\end{equation}
The conditions (\ref{cv10}) and (\ref{cv11}) follow from the bound (\ref{pt6}).  

Now assume that $1 \le x = \Re(z)$.  We have
\begin{align*}\label{pt7}
\begin{split}
\bigl|F_{\mu}(x + iy)\bigr| &= \left|\int_1^x F_{\mu}^{\prime}(u)\ \du + i\int_0^y F_{\mu}^{\prime}(x + iv)\ \dv\right| \\
		 &\le \int_1^x \bigr|F_{\mu}^{\prime}(u)\bigr|\ \du + |y|\bigl|F_{\mu}^{\prime}(x)\bigr|,
\end{split}
\end{align*}
and therefore
\begin{equation}\label{pt8}
\left|\frac{F_{\mu}(x + iy)}{x + iy}\right|
	\le \frac{1}{x}\int_1^x \bigr|F_{\mu}^{\prime}(u)\bigr|\ \du + \bigl|F_{\mu}^{\prime}(x)\bigr|.
\end{equation} 
Then (\ref{pt3}) and (\ref{pt8}) imply that
\begin{equation*}\label{pt9}
\lim_{x\rightarrow \infty} \left|\frac{F_{\mu}(x + iy)}{x+iy}\right| = 0
\end{equation*}
uniformly in $y$.  The remaining condition (\ref{cv12}) follows from this.
\end{proof}

We are now in position to apply the results of section 2 and section 3 to the function $F_{\mu}(z)$.
In view of the identities (\ref{pt4}), the entire function $G_{\mu}(z)$, defined by (\ref{intro36}), and the entire
function $\G(F_{\mu}, z)$, defined by (\ref{cv34}), are equal.  If $0 < \beta < \h$, and $\beta < \Re(z)$, then 
from (\ref{cv41}) of Lemma \ref{lem2.4} we have
\begin{equation}\label{pt10}
F_{\mu}(z) - G_{\mu}(z) = I(\beta, F_{\mu}; z).
\end{equation}
Applying Lemma \ref{lem2.6} we conclude that $G_{\mu}(z)$ is an entire function of exponential type at most $2\pi$.
This verifies (i) in the statement of Theorem \ref{thm1.1}.

Next we define a sequence of measures $\nu_1, \nu_2, \nu_3, \dots $ on Borel subsets $E\subseteq (0,\infty)$ by
\begin{equation}\label{pt20}
\nu_n(E) = \int_E \bigl(e^{-\lambda/n} - e^{-\lambda n}\bigr)\ \dmu,\quad\text{for}\quad n = 1, 2, \dots .
\end{equation}
Then
\begin{align*}
\nu_n\{(0,\infty)\} &= \int_0^{\infty} \int_{1/n}^n \lambda e^{-\lambda u}\ \du\ \dmu \\
	&= - \int_{1/n}^n F_{\mu}^{\prime}(u)\ \du \\
	&= F_{\mu}(1/n) - F_{\mu}(n) < \infty, 
\end{align*}
and therefore $\nu_n$ is a finite measure for each $n$.  It will be convenient to simplify (\ref{finite10}) and
(\ref{def0}).  For $z$ in $\C$ and $n$ a positive integer we write
\begin{equation}\label{pt21}
L_n(z) =\int_{0}^{\infty} L(\lambda,z)\ \dnnu,
\end{equation} 
and for $z$ in $\rr$ we write
\begin{equation}\label{pt22}
\Psi_n(z) = \int_0^{\infty} e^{-\lambda z}\ \dnnu.
\end{equation}
It follows from Lemma \ref{lem3.4} that $L_n(z)$ is an entire function of exponential type at most $2\pi$.
If $0 < \beta < \h$ then (\ref{cv48}) and (\ref{cv6}) imply that
\begin{equation}\label{pt23}
\Psi_n(z) - L_n(z) = I(\beta, \Psi_n; z) = I\bigl(\beta, \Psi_n - \Psi_n(1); z\bigr)
\end{equation}
for all complex $z$ such that $\beta < \Re(z)$.  From the definitions (\ref{pt20}), (\ref{pt21}), and 
(\ref{pt22}), we find that
\begin{equation}\label{pt24}
\Psi_n(x) - L_n(x) 
	= \int_0^{\infty} \bigl(e^{-\lambda x} - L(\lambda, x)\bigr)\bigl(e^{-\lambda/n} - e^{-\lambda n}\bigr)\ \dmu
\end{equation}
for all positive real $x$.   

Let $w = u + iv$ be a point in $\rr$.  Then
\begin{equation}\label{pt26}
\Psi_n(w) - \Psi_n(1) = \int_0^{\infty}\bigl(e^{-\lambda w} - e^{-\lambda}\bigr)
		\bigl(e^{-\lambda/n} - e^{-\lambda n}\bigr)\ \dmu,
\end{equation}
and
\begin{equation*}\label{pt27}
\bigl|e^{-\lambda/n} - e^{-\lambda n}\bigr| \le 1
\end{equation*}
for all positive real $\lambda$ and positive integers $n$.  We let $n\rightarrow \infty$ on both sides
of (\ref{pt26}) and apply the dominated convergence theorem.  In this way we conclude that
\begin{equation}\label{pt28}
\lim_{n\rightarrow \infty} \Psi_n(w) - \Psi_n(1) = F_{\mu}(w)
\end{equation}
at each point $w$ in $\rr$.  If $0 < \beta < \h$ then, as in the 
proof of Lemma \ref{lem4.1}, on the line $\beta = \Re(w)$ we have
\begin{align*}\label{pt28}
\bigl|\Psi_n(w) - \Psi_n(1)\bigr| &\le \int_0^{\infty}\left|\int_1^w \lambda e^{-\lambda t}\ \dt\right|\ \dmu \\
	 &\le (|w| + 1)\bigl|F_{\mu}^{\prime}(\beta)\bigr|.
\end{align*}
It follows that
\begin{equation*}\label{pt30}
\left|\frac{\Psi_n(w) - \Psi_n(1)}{w}\right|
\end{equation*}
is bounded on the line $\beta = \Re(w)$.  From this observation, together with (\ref{pt23}) and (\ref{pt28}),
we conclude that
\begin{align}\label{pt31}
\begin{split}
\lim_{n\rightarrow \infty} \Psi_n(z) - L_n(z) &= \lim_{n\rightarrow \infty} I(\beta, \Psi_n - \Psi_n(1); z) \\ 
	&= I(\beta, F_{\mu}; z) \\ 
	&= F_{\mu}(z) - G_{\mu}(z)
\end{split} 
\end{align}
at each complex number $z$ with $\beta < \Re(z)$.  In particular, we have
\begin{equation}\label{pt32}
\lim_{n\rightarrow \infty} \Psi_n(x) - L_n(x) = F_{\mu}(x) - G_{\mu}(x)
\end{equation}
for all positive $x$.  We combine (\ref{pt24}), (\ref{pt32}), and use the monotone convergence theorem.  This leads
to the identity
\begin{equation}\label{pt33}
F_{\mu}(x) - G_{\mu}(x) = \int_0^{\infty} \bigl(e^{-\lambda x} - L(\lambda, x)\bigr)\ \dmu
\end{equation}
for all positive $x$.  Then we use the identity on the left of (\ref{pt4}), and the fact that $x\mapsto G_{\mu}(x)$
is an even function, to write (\ref{pt33}) as
\begin{equation}\label{pt34}  
f_{\mu}(x) - G_{\mu}(x) = \int_0^{\infty} \bigl(e^{-\lambda |x|} - L(\lambda, x)\bigr)\ \dmu
\end{equation}
for all $x\not= 0$.  If $f_{\mu}(0)$ is finite then (\ref{pt34}) holds at $x = 0$ by continuity.  If 
$f_{\mu}(0) = \infty$ then both sides of (\ref{pt34}) are $\infty$.  And (\ref{intro6}) implies that (\ref{pt34})
is nonnegative for all real $x$.  This establishes both (ii) and (iii) in the statement of Theorem \ref{thm1.1}. 

Because the integrand on the right of (\ref{pt34}) is nonnegative, we get
\begin{align}\label{pt40}
\begin{split}
\int_{-\infty}^{\infty} \big\{f_{\mu}(x) - G_{\mu}(x)\big\}\ \dx
 	&= \int_{-\infty}^{\infty} \int_0^{\infty} \bigl(e^{-\lambda |x|} - L(\lambda, x)\bigr)\ \dmu \dx \\
	&= \int_0^{\infty} \int_{-\infty}^{\infty} \bigl(e^{-\lambda |x|} - L(\lambda, x)\bigr)\ \dx \dmu \\
	&= \int_0^{\infty} \big\{\tfrac{2}{\lambda} - \csch\left(\tfrac{\lambda}{2}\right)\big\}\ \dmu.
\end{split}
\end{align}
by Fubini's theorem.  This proves (iv) in the statement of Theorem \ref{thm1.1}.  Similarly, if $t\not= 0$ we find that
\begin{align}\label{pt41}
\begin{split}
\int_{-\infty}^{\infty} &\big\{f_{\mu}(x) - G_{\mu}(x)\big\}e(-tx)\ \dx \\
	&= \int_{-\infty}^{\infty} \Bigg\{\int_0^{\infty} \bigl(e^{-\lambda |x|} 
		- L(\lambda, x)\bigr)\ \dmu\Bigg\}e(-tx) \ \dx \\
	&= \int_0^{\infty} \Bigg\{\int_{-\infty}^{\infty} \bigl(e^{-\lambda |x|} 
		- L(\lambda, x)\bigr)e(-tx) \ \dx\Bigg\}\ \dmu \\
	&= \int_0^{\infty} \Big\{\frac{2\lambda}{\lambda^2 + 4\pi^2t^2}\Big\}\ \dmu
	 	- \int_0^{\infty} \tL(\lambda, t)\ \dmu.
\end{split}
\end{align}
This proves (v) in Theorem \ref{thm1.1}.

Finally, we assume that $\tG(z)$ is a real entire function of exponential type at most $2\pi$ such that
\begin{equation}\label{pt44}
\tG(x) \le f_{\mu}(x)
\end{equation}
for all real $x$.  Obviously (\ref{intro45}) is trivial if the integral on the right of (\ref{intro45}) is infinite.
Hence we may assume that
\begin{equation}\label{pt45}
\int_{-\infty}^{\infty} \left\{f_{\mu}(x) - \tG(x)\right\}\ \dx < \infty.
\end{equation}
Then (\ref{intro45}) is equivalent to 
\begin{equation}\label{pt46}
0 \le \int_{-\infty}^{\infty} \left\{G_{\mu}(x) - \tG(x)\right\}\ \dx. 
\end{equation}
As $G_{\mu}(z) - \tG(z)$ is a real entire function of exponential type at most $2\pi$ and
is integrable on $\R$, we can apply \cite[Lemma 4]{GV}.  By that result we get
\begin{align}\label{pt47}
\begin{split}
\lim_{N\rightarrow \infty} \sum_{n=-N}^N \left(1 - \frac{|n|}{N}\right)\big\{G_{\mu}(n - \hh)&- \tG(n - \hh)\big\} \\
	&= \int_{-\infty}^{\infty} \left\{G_{\mu}(x) - \tG(x)\right\}\ \dx.
\end{split}
\end{align}
It follows from (\ref{intro37}) and (\ref{pt44}) that
\begin{equation}\label{pt48}
0 \le G_{\mu}(n - \hh) - \tG(n - \hh)
\end{equation}
for each integer $n$.  Therefore (\ref{pt47}) and (\ref{pt48}) imply that the integral (\ref{pt46}) is nonnegative.  This
proves (vi) in the statement of Theorem \ref{thm1.1}.
If the value of the integral (\ref{pt46}) is zero, then we have
\begin{equation*}\label{pt49}
0 = G_{\mu}(n - \hh) - \tG(n - \hh)
\end{equation*}
for each integer $n$.  It follows that
\begin{equation*}\label{pt50}
G_{\mu}(n - \hh) = \tG(n - \hh) = f_{\mu}(n - \hh)
\end{equation*}
at each integer $n$.  As both $G_{\mu}(x) \le f_{\mu}(x)$ and $\tG(x) \le f_{\mu}(x)$ for all real $x$, we find that
\begin{equation}\label{pt51}
G_{\mu}^{\prime}(n - \hh) = \tG^{\prime}(n - \hh) = f_{\mu}^{\prime}(n - \hh)
\end{equation}
for each integer $n$.   A second application of \cite[Lemma 4]{GV} shows that $G_{\mu}(z) = \tG(z)$ for all complex $z$.
This completes the proof of (vii) in Theorem \ref{thm1.1}.


\section{Proof of Theorem \ref{thm1.2}}

Let $\mu$ be a measure defined on the Borel subsets of $(0,\infty)$ that satisfies (\ref{intro47}). We keep here the 
same notation used in the proof of Theorem \ref{thm1.1}. Observe that the entire function $\sH(F_{\mu},z)$ defined in 
(\ref{cv36}) and the function $H_{\mu}(z)$ defined by (\ref{intro49}) satisfy
\begin{equation}\label{Sec5.1}
H_{\mu}(z) = \sH(F_{\mu},z) + \left(\dfrac{\sin \pi z}{\pi z} \right)^2 f_{\mu}(0)
\end{equation}
It follows from (\ref{Sec5.1}) and Lemma \ref{lem2.6} that $H_{\mu}(z)$ is an entire function of exponential type at 
most $2\pi$. This verifies (i) in the statement of Theorem \ref{thm1.2}. If $0< \beta < 1$ and $\beta < \Re(z)$, then 
from (\ref{Sec5.1}), (\ref{cv43}) of Lemma \ref{lem2.4} and (\ref{cv49}) we have
\begin{equation}\label{Sec5.2}
H_{\mu}(z) - F_{\mu}(z) = J(\beta, f_{\mu}(0) - F_{\mu};z)
\end{equation}
For the measures $\nu_n$ defined in (\ref{pt20}) we write
\begin{equation*}\label{Sec5.3}
a_n = \nu_n\{(0,\infty)\}
\end{equation*}
For $z \in \C$ we also define
\begin{equation}\label{Sec5.4}
M_n(z) =\int_{0}^{\infty} M(\lambda,z)\ \dnnu, 
\end{equation}
which is an entire function of exponential type at most $2\pi$ by Lemma \ref{lem3.4}. If $0< \beta < 1$ and 
$\beta < \Re(z)$, from (\ref{pt22}) and (\ref{cv7}) we have
\begin{equation}\label{Sec5.5}
M_n(z) - \Psi_n(z) = J(\beta, a_n - \Psi_n; z). 
\end{equation}
Let $w = u + i v$ be a point in $\rr$. Then
\begin{equation}\label{Sec5.7}
a_n - \Psi_n(w) = \int_0^{\infty}\bigl( 1 - e^{-\lambda w}\bigr)
		\bigl(e^{-\lambda/n} - e^{-\lambda n}\bigr)\ \dmu.
\end{equation}
Since $\bigl|\,e^{-\lambda/n} - e^{-\lambda n}\bigr| \leq 1$, by dominated convergence we have 
\begin{equation}\label{Sec5.8}
\lim_{n\rightarrow \infty} a_n - \Psi_n(w) = f_{\mu}(0) - F_{\mu}(w).
\end{equation}
If $0< \beta<1$, then on the line $\beta = \Re(w)$ we have from (\ref{Sec5.7})
\begin{align*}\label{Sec5.9}
\bigl|a_n - \Psi_n(w)\bigr| &\le \int_0^{\infty}\left|\int_0^w \lambda e^{-\lambda s}\ \ds\right|\ \dmu \\
   & \le \int_0^{\infty}\left\{ \left|\int_0^{\beta} \lambda e^{-\lambda s}\ \ds\right|\ 
		+ \left|\int_{\beta}^{\beta + i v} \lambda e^{-\lambda s}\ \ds\right|\ \right\}\dmu \\
   & \le \int_0^{\beta} \int_0^{\infty} \lambda e^{-\lambda s}\ \dmu \ \ds + |v| \int_0^{\infty}\lambda e^{-\lambda \beta}\ \dmu \\
   & = -\int_0^{\beta} F_{\mu}^{\prime}(s)\ \ds + |v|\bigl|F_{\mu}^{\prime}(\beta)\bigr| \\
	 & = f_{\mu}(0) - F_{\mu}(\beta) + |v|\bigl|F_{\mu}^{\prime}(\beta)\bigr|.
\end{align*}
It follows that 
\begin{equation*}\label{Sec5.10}
\left|\frac{a_n - \Psi_n(w)}{w}\right|
\end{equation*}
is bounded on the line $\beta = \Re(w)$. From this observation, (\ref{Sec5.5}), and (\ref{Sec5.8}), we conclude that 
\begin{align}\label{Sec5.11}
\begin{split}
\lim_{n\rightarrow \infty} M_n(z) - \Psi_n(z)  &= \lim_{n\rightarrow \infty} J(\beta, a_n - \Psi_n; z) \\ 
	&= J(\beta, f_{\mu}(0) - F_{\mu}; z) \\ 
	&= H_{\mu}(z) - F_{\mu}(z)
\end{split} 
\end{align}
for each complex number $\beta < \Re(z)$.  As 
\begin{equation}\label{Sec5.12}
M_n(x) - \Psi_n(x) 
	= \int_0^{\infty} \bigl( M(\lambda, x)- e^{-\lambda x} \bigr)\bigl(e^{-\lambda/n} - e^{-\lambda n}\bigr)\ \dmu
\end{equation}
for all positive real $x$, the monotone convergence theorem, together with (\ref{Sec5.11}), leads to the identity
\begin{equation}\label{Sec5.13}
H_{\mu}(x) - F_{\mu}(x) = \int_0^{\infty} \bigl( M(\lambda, x)- e^{-\lambda x} \bigr)\ \dmu
\end{equation}
for all positive $x$.  Then we use the identity on the left of (\ref{pt4}), and the fact that $x\mapsto H_{\mu}(x)$
is an even function, to write (\ref{Sec5.13}) as
\begin{equation}\label{Sec5.14}  
H_{\mu}(x) - f_{\mu}(x) = \int_0^{\infty} \bigl( M(\lambda, x)- e^{-\lambda |x|} \bigr)\ \dmu
\end{equation}
for all $x\neq 0$. At $x=0$ both sides of (\ref{Sec5.14}) are zero. From (\ref{intro6}) we conclude that (\ref{Sec5.14}) 
is nonnegative for all real $x$.  This establishes both (ii) and (iii) in the statement of Theorem \ref{thm1.2}.

The proofs of parts (iv)-(vii) of Theorem \ref{thm1.2} are similar to the corresponding versions for Theorem \ref{thm1.1}. 
There is just one detail in the proof of part (vii) that we should point out. When considering the case of equality 
in (\ref{intro57}) one shows that
\begin{equation*} 
H_{\mu}(n) = \tH(n) = f_{\mu}(n)
\end{equation*}
at each integer $n$. The fact that both  $H_{\mu}(x) \geq f_{\mu}(x)$ and $\tH_{\mu}(x) \geq f_{\mu}(x)$ for all real $x$ 
is sufficient to conclude that 
\begin{equation*}
H_{\mu}^{\prime}(n) = \tH^{\prime}(n) = f_{\mu}^{\prime}(n)
\end{equation*}
at each nonzero integer $n$, since $f_{\mu}$ is not necessarily differentiable at $x=0$. However, an application 
of \cite[Lemma 4, equation 2.3]{GV}  allows us to conclude that 
\begin{equation*}
H_{\mu}^{\prime}(0) = \tH^{\prime}(0). 
\end{equation*}
A further application of \cite[Lemma 4]{GV} proves that $H_{\mu}(z) = \tH(z)$ for all complex $z$. 

\section{Extremal trigonometric polynomials}

We consider the problem of majorizing and minorizing certain real valued periodic functions by real
trigonometric polynomials of bounded degree.  We identify functions defined on $\R$ and having period $1$
with functions defined on the compact quotient group $\R/\Z$.  For real numbers $x$ we write
\begin{equation*}
\|x\| = \min\{|x - m|: m\in\Z\}
\end{equation*}
for the distance from $x$ to the nearest integer.  Then $\|\ \|:\R/\Z\rightarrow [0,\h]$ is well defined, and
$(x,y)\rightarrow \|x - y\|$ defines a metric on $\R/\Z$ which induces its quotient topology.  Integrals over 
$\R/\Z$ are with respect to Haar measure normalized so that $\R/\Z$ has measure $1$.  

Let $F:\C\rightarrow \C$ be an entire function of exponential type at most $2\pi\delta$, where $\delta$ is a 
positive parameter, and assume that $x\mapsto F(x)$ is integrable on $\R$.  Then the Fourier transform
\begin{equation}\label{ef0}
\tF(t) = \int_{-\infty}^{\infty} F(x)e(-tx)\ \dx
\end{equation}
is a continuous function on $\R$.  By classical results of Plancherel and Polya \cite{PP} (see also 
\cite[Chapter 2, Part 2, section 3]{Y}) we have
\begin{equation}\label{ef1}
\sum_{m=-\infty}^{\infty} |F(\alpha_m)| \le C_1(\epsilon, \delta) \int_{-\infty}^{\infty} |F(x)|\ \dx,
\end{equation}
where $m\mapsto \alpha_m$ is a sequence of real numbers such that $\alpha_{m+1} - \alpha_m \ge \epsilon > 0$, and 
\begin{equation}\label{ef2}
\int_{-\infty}^{\infty} |F^{\prime}(x)|\ \dx \le C_2(\delta) \int_{-\infty}^{\infty} |F(x)|\ \dx.
\end{equation}
Plainly (\ref{ef1}) implies that $F$ is uniformly bounded on $\R$, and therefore $x\mapsto |F(x)|^2$ is integrable.
Then it follows from the Paley-Wiener theorem (see \cite[Theorem 19.3]{Rudin}) that $\tF(t)$ is supported
on the interval $[-\delta, \delta]$.  

The bound (\ref{ef2}) implies that $x\mapsto F(x)$ has bounded variation on $\R$.  Therefore the Poisson 
summation formula (see \cite[Volume I, Chapter 2, section 13]{Z}) holds as a pointwise identity
\begin{equation}\label{ef3}
\sum_{m=-\infty}^{\infty} F(x + m) = \sum_{n=-\infty}^{\infty} \tF(n) e(nx),
\end{equation}
for all real $x$.  It follows from (\ref{ef1}) that the sum on the left of (\ref{ef3}) is absolutely
convergent.  As the continuous function $\tF(t)$ is supported on $[-\delta, \delta]$, the sum
on the right of (\ref{ef3}) has only finitely many nonzero terms, and so defines a trigonometric polynomial
in $x$.

Next we consider the entire functions $z\mapsto L\bigl(\delta^{-1}\lambda, \delta z\bigr)$ and 
$z\mapsto M\bigl(\delta^{-1}\lambda, \delta z\bigr)$.  These functions have exponential type at most $2\pi\delta$.
Therefore we apply (\ref{ef3}) and obtain the identities
\begin{equation}\label{ef4}
\sum_{m=-\infty}^{\infty} L\bigl(\delta^{-1}\lambda, \delta(x + m)\bigr) 
	= \delta^{-1} \sum_{|n| \le \delta} \tL\bigl(\delta^{-1}\lambda, \delta^{-1}n\bigr) e(nx)
\end{equation}
and
\begin{equation}\label{ef5}
\sum_{m=-\infty}^{\infty} M\bigl(\delta^{-1}\lambda, \delta(x + m)\bigr) 
	= \delta^{-1} \sum_{|n| \le \delta} \tM\bigl(\delta^{-1}\lambda, \delta^{-1}n\bigr) e(nx)
\end{equation}
for all real $x$, and for all positive values of the parameters $\delta$ and $\lambda$.  For our purposes it will 
be convenient to use (\ref{ef4}) and (\ref{ef5}) with $\delta = N+1$, where $N$ is a nonnegative integer, and to modify
the constant term.  Therefore we define trigonometric polynomials of degree $N$ by
\begin{align}\label{ef6}
\begin{split}
l(\lambda, N; x) &= -\tfrac{2}{\lambda} + \tfrac{1}{N+1} \sum_{n=-N}^N \tL\bigl(\tfrac{\lambda}{N+1}, \tfrac{n}{N+1}\bigr) e(nx) \\
	 &= -\big\{\tfrac{2}{\lambda} - \tfrac{1}{N+1}\csch\bigl(\tfrac{\lambda}{2N+2}\bigr)\big\} + 
		\tfrac{1}{N+1} \sum_{1\le |n|\le N} \tL\bigl(\tfrac{\lambda}{N+1}, \tfrac{n}{N+1}\bigr) e(nx),	 
\end{split}
\end{align}
and
\begin{align}\label{ef7}
\begin{split}
m(\lambda, N; x) &= -\tfrac{2}{\lambda} + \tfrac{1}{N+1}\sum_{n=-N}^N \tM\bigl(\tfrac{\lambda}{N+1}, \tfrac{n}{N+1}\bigr) e(nx) \\
	 &= \big\{\tfrac{1}{N+1}\coth\bigl(\tfrac{\lambda}{2N+2}\bigr) - \tfrac{2}{\lambda}\big\} + 
		\tfrac{1}{N+1} \sum_{1\le |n|\le N} \tM\bigl(\tfrac{\lambda}{N+1}, \tfrac{n}{N+1}\bigr) e(nx).	 
\end{split}
\end{align}
We note that
\begin{equation}\label{ef8}
\int_{\R/\Z} l(\lambda, N; x)\ \dx = -\big\{\tfrac{2}{\lambda} 
	- \tfrac{1}{N+1}\csch\bigl(\tfrac{\lambda}{2N+2}\bigr)\big\} < 0, 
\end{equation}
and 
\begin{equation}\label{ef9}
\int_{\R/\Z} m(\lambda, N; x)\ \dx = \big\{\tfrac{1}{N+1}\coth\bigl(\tfrac{\lambda}{2N+2}\bigr) 
	- \tfrac{2}{\lambda}\big\} > 0.
\end{equation}

For $0 < \lambda$ the function $x\mapsto e^{-\lambda |x|}$ is continuous, integrable on $\R$, and has
bounded variation.  Therefore the Poisson summation formula also provides the pointwise identity
\begin{equation}\label{ef10}
\sum_{m=-\infty}^{\infty} e^{-\lambda |x + m|} 
	= \sum_{n=-\infty}^{\infty} \frac{2\lambda}{\lambda^2 + 4\pi^2n^2}e(nx).
\end{equation}
And we find that
\begin{equation}\label{ef11}
\sum_{m=-\infty}^{\infty} e^{-\lambda |x + m|} 
	= \frac{\cosh\bigl(\lambda(x - [x] - \hh)\bigr)}{\sinh\bigl(\tfrac{\lambda}{2}\bigr)},
\end{equation}
where $[x]$ is the integer part of the real number $x$.  For our purposes it will be convenient to define
\begin{equation*}\label{ef12}
p:(0,\infty)\times \R/\Z \rightarrow \R 
\end{equation*}
by
\begin{equation}\label{ef13}
p(\lambda, x) = - \tfrac{2}{\lambda} + \sum_{m=-\infty}^{\infty} e^{-\lambda |x + m|}.
\end{equation}
Then $p(\lambda, x)$ is continuous on $(0,\infty)\times \R/\Z$, and differentiable with respect to $x$ at 
each noninteger point $x$.  It follows from (\ref{ef10}) that the Fourier coefficients of 
$x\mapsto p(\lambda, x)$ are given by
\begin{equation}\label{ef15}
\int_{\R/\Z} p(\lambda, x)\ \dx = 0,
\end{equation}
and
\begin{equation}\label{ef16}
\int_{\R/\Z} p(\lambda, x)e(-nx)\ \dx = \frac{2\lambda}{\lambda^2 + 4\pi^2n^2}
\end{equation}
for integers $n\not= 0$. 

\begin{theorem}\label{thm6.1}  Let $\lambda$ be a positive real number and $N$ a nonnegative integer. 
\begin{itemize}
\item[(i)]  The inequality
\begin{equation}\label{ef21}
l(\lambda, N; x) \le p(\lambda, x) \le m(\lambda, N; x)
\end{equation}
holds at each point $x$ in $\R/\Z$.  
\item[(ii)]  There is equality in the inequality on the left of {\rm (\ref{ef21})} for
\begin{equation}\label{ef22}
x = \tfrac{n - \h}{N+1}\quad\text{and}\quad n = 1, 2, \dots , N+1,
\end{equation}
and there is equality in the inequality on the right of {\rm (\ref{ef21})} for
\begin{equation}\label{ef23}
x = \tfrac{n}{N+1}\quad\text{and}\quad n = 1, 2, \dots , N+1.
\end{equation}
\item[(iii)]  If $\wl(x)$ is a real trigonometric polynomial of degree at most $N$ such that 
\begin{equation*}
\wl(x) \le p(\lambda, x)
\end{equation*}
at each point $x$ in $\R/\Z$, then
\begin{equation}\label{ef24}
\int_{\R/\Z} \wl(x)\ \dx \le \int_{\R/\Z} l(\lambda, N; x)\ \dx.
\end{equation}
\item[(iv)]  If $\wm(x)$ is a real trigonometric polynomial of degree at most $N$ such that 
\begin{equation*}
p(\lambda, x) \le \wm(x)
\end{equation*}
at each point $x$ in $\R/\Z$, then
\begin{equation}\label{ef25}
\int_{\R/\Z} m(\lambda, N; x)\ \dx \le \int_{\R/\Z} \wm(x)\ \dx.
\end{equation}
\item[(v)]  There is equality in the inequality {\rm (\ref{ef24})} if and only if $\wl(x) = l(\lambda, N; x)$,
and there is equality in the inequality {\rm (\ref{ef25})} if and only if $\wm(x) = m(\lambda, N; x)$.
\end{itemize}
\end{theorem}

\begin{proof}  From the inequality (\ref{intro6}) we have
\begin{equation}\label{ef28}
L\bigl(\tfrac{\lambda}{N+1}, (N+1)|x + m|\bigr) \le e^{-\lambda |x + m|} \le M\bigl(\tfrac{\lambda}{N+1}, (N+1)|x + m|\bigr)
\end{equation}
for all real $x$ and integers $m$.  We sum (\ref{ef28}) over integers $m$ in $\Z$, and use (\ref{ef4}) and (\ref{ef5}) with 
$\delta = N+1$.  Then (\ref{ef21}) follows from the definitions (\ref{ef6}), (\ref{ef7}), and (\ref{ef13}).

It follows from (\ref{gv3}) that the entire function $z\mapsto L(\lambda, z)$ interpolates the values of
$x\mapsto e^{-\lambda |x|}$ at real numbers $x$ such that $x + \h$ is an integer.  That is, there is equality in 
the inequality
\begin{equation*}
L(\lambda, x) \le e^{-\lambda |x|}
\end{equation*}
whenever $x = n - \h$ with $n$ in $\Z$.  Hence there is equality in the inequality 
\begin{equation*}\label{ef29}
L\bigl(\tfrac{\lambda}{N+1}, (N+1)|x + m|\bigr) \le e^{-\lambda |x + m|} 
\end{equation*}
whenever $x$ has the form indicated in (\ref{ef22}) and $m$ is an integer.  This implies that there is equality in the
inequality on the left of (\ref{ef21}) when $x$ has the form (\ref{ef22}).

In a similar manner, it follows from (\ref{gv4}) that there is
equality in the inequality 
\begin{equation*}\label{ef30}
e^{-\lambda |x|} \le M(\lambda, x)
\end{equation*}
whenever $x = n$ with $n$ in $\Z$.  Hence there is equality in the inequality
\begin{equation*}\label{ef31}
e^{-\lambda |x + m|} \le M\bigl(\tfrac{\lambda}{N+1}, (N+1)|x + m|\bigr)
\end{equation*}
whenever $x$ has the form indicated in (\ref{ef23}) and $m$ is an integer.  This leads to the conclusion that
there is equality in the inequality on the right of (\ref{ef21}) when $x$ has the form (\ref{ef23}).

Now suppose that $\wl(x)$ is a real trigonometric polynomial of degree at most $N$ such that 
\begin{equation*}\label{ef32}
\wl(x) \le p(\lambda, x)
\end{equation*}
at each point $x$ in $\R/\Z$.  Using the case of equality in the inequality on the left of (\ref{ef21}),
we get
\begin{align}\label{ef33}
\begin{split}
\int_{\R/\Z} \wl(x)\ \dx &= \tfrac{1}{N+1} \sum_{n=1}^{N+1} \wl\bigl(\tfrac{n-\h}{N+1}\bigr)
	\le \tfrac{1}{N+1} \sum_{n=1}^{N+1} p\bigl(\lambda, \tfrac{n-\h}{N+1}\bigr) \\
	&= \tfrac{1}{N+1} \sum_{n=1}^{N+1} l\bigl(\lambda, N; \tfrac{n-\h}{N+1}\bigr)
	= \int_{\R/\Z} l(\lambda, N; x)\ \dx.
\end{split}
\end{align}
This proves the inequality (\ref{ef24}), and the same sort of argument can be used to prove (\ref{ef25}).

If there is equality in (\ref{ef24}), then it is clear that there is equality in (\ref{ef33}).  This implies
that
\begin{equation*}\label{ef34}
\wl\bigl(\tfrac{n-\h}{N+1}\bigr) = l\bigl(\lambda, N; \tfrac{n-\h}{N+1}\bigr)
\end{equation*}  
for $n = 1, 2, \dots , N+1$.  As both $\wl(x)$ and $l(\lambda, N; x)$ are less than or
equal to $p(\lambda, x)$ at each point $x$ of $\R/\Z$, we also conclude that
\begin{equation*}\label{ef35}
\wl^{\prime}\bigl(\tfrac{n-\h}{N+1}\bigr) = l^{\prime}\bigl(\lambda, N; \tfrac{n-\h}{N+1}\bigr)
\end{equation*} 
for each $n = 1, 2, \dots , N+1$.  This shows that the real trigonometric polynomial
\begin{equation}\label{ef36}
l(\lambda, N; x) - \wl(x)
\end{equation}
has degree at most $N$, and it each point $x = \frac{n-\h}{N+1}$, where $n = 1, 2, \dots , N+1$, the polynomial
and its derivative both vanish.  It is well known (see \cite[Vol. II, page 23]{Z}) that such a trigonometric polynomial
must be identically zero.  In a similar manner, if equality occurs in the inequality (\ref{ef25}), then we find
that
\begin{equation*}\label{ef37}
\wm(x) - m(\lambda, N; x)
\end{equation*}
is identically zero.  This completes the proof of assertion (v) in the statement of the Theorem.
\end{proof}

It follows from (\ref{ef11}) and (\ref{ef13}) that
\begin{equation}\label{ef42}
-\big\{\tfrac{2}{\lambda} - \csch\bigl(\tfrac{\lambda}{2}\bigr)\big\} = p(\lambda, \hh) \le p(\lambda, x)
	\le p(\lambda, 0) = \coth\bigl(\tfrac{\lambda}{2}\bigr) - \tfrac{2}{\lambda}.
\end{equation}
Then (\ref{ef42}) provides the useful inequality
\begin{align}\label{ef43}
\begin{split}
\bigl|p(\lambda, x)\bigr| &\le \bigl|p(\lambda, x) - p(\lambda, \hh)\bigr| + \bigl|p(\lambda, \hh)\bigr| \\
	&= p(\lambda, x) - p(\lambda, \hh) - p(\lambda, \hh) \\
	&= p(\lambda, x) + 2\big\{\tfrac{2}{\lambda} - \csch\bigl(\tfrac{\lambda}{2}\bigr)\big\}
\end{split}
\end{align}
at each point $(\lambda, x)$ in $(0, \infty)\times \R/\Z$.  From (\ref{ef15}) and (\ref{ef43}) we conclude that
\begin{equation}\label{ef44}
\int_{\R/\Z} \bigl|p(\lambda, x)\bigr|\ \dx \le 2\big\{\tfrac{2}{\lambda} - \csch\bigl(\tfrac{\lambda}{2}\bigr)\big\}. 
\end{equation}

Let $\mu$ be a measure on the Borel subsets of $(0,\infty)$ that satisfies (\ref{intro31}).  For $0 < x < 1$ it 
follows from (\ref{ef11}) and (\ref{ef13}) that $\lambda\mapsto p(\lambda, x)$ is integrable on $(0,\infty)$ 
with respect to $\mu$.  We define $q_{\mu}:\R/\Z\rightarrow \R\cup\{\infty\}$ by
\begin{equation}\label{ef50}
q_{\mu}(x) = \int_0^{\infty} p(\lambda, x)\ \dmu,
\end{equation}
where
\begin{equation}\label{ef51}
q_{\mu}(0) = \int_0^{\infty}\big\{\coth\bigl(\tfrac{\lambda}{2}\bigr) - \tfrac{2}{\lambda}\big\}\ \dmu
\end{equation}
may take the value $\infty$.  Using (\ref{ef44}) and Fubini's theorem we have
\begin{align*}\label{ef52}
\begin{split}
\int_{\R/\Z} \bigl|q_{\mu}(x)\bigr|\ \dx &\le \int_0^{\infty} \int_{\R/\Z} \bigl|p(\lambda, x)\bigr|\ \dx\ \dmu \\ 
	&\le 2 \int_0^{\infty} \big\{\tfrac{2}{\lambda} - \csch\bigl(\tfrac{\lambda}{2}\bigr)\big\}\ \dmu < \infty,
\end{split}
\end{align*}
so that $q_{\mu}$ is integrable on $\R/\Z$.  Using (\ref{ef15}) and (\ref{ef16}), we find that the Fourier 
coefficients of $q_{\mu}$ are given by
\begin{equation}\label{ef53}
\tq_{\mu}(0) = \int_{\R/\Z} q_{\mu}(x)\ \dx = \int_0^{\infty} \int_{\R/\Z} p(\lambda, x)\ \dx\ \dmu = 0,
\end{equation}
and
\begin{align}\label{ef54}
\begin{split}
\tq_{\mu}(n) &= \int_{R/\Z} q_{\mu}(x) e(-nx)\ \dx \\
	     &= \int_0^{\infty} \int_{\R/\Z} p(\lambda, x)e(-nx)\ \dx\ \dmu \\
 	     &= \int_0^{\infty} \frac{2\lambda}{\lambda^2 + 4\pi^2n^2}\ \dmu,
\end{split}
\end{align}
for integers $n\not= 0$.  As $n\mapsto \tq_{\mu}(n)$ is an even function of $n$, and $\tq_{\mu}(n) \ge \tq_{\mu}(n+1)$ 
for $1 \le n$, the partial sums
\begin{equation}\label{ef55}
q_{\mu}(x) = \lim_{N\rightarrow \infty} \sum_{\substack{n=-N\\ n\not= 0}}^N \tq_{\mu}(n)e(nx)
\end{equation}
converge uniformly on compact subsets of $\R/\Z\setminus\{0\}$, (see \cite[Chapter I, Theorem 2.6]{Z}).
In particular, $q_{\mu}(x)$ is continuous on $\R/\Z\setminus\{0\}$. 

Next we define the function
\begin{equation*}\label{Sec6.1}
j:(0,\infty)\times \R/\Z \rightarrow \R 
\end{equation*}
by $j(\lambda,x) = 0$ if $x$ is in $\Z$, and 
\begin{equation}\label{Sec6.2}
j(\lambda,x) = \dfrac{\partial p}{\partial x}(\lambda,x)
	= \frac{\lambda \sinh\bigl(\lambda(x - [x] - \hh)\bigr)}{\sinh\bigl(\tfrac{\lambda}{2}\bigr)},
\end{equation}
if $x$ is not in $\Z$. We note that $j(\lambda,x)$ satisfies the elementary inequality
\begin{equation}\label{Sec6.2.1}
\bigl|j(\lambda,x)\bigr| \le \lambda e^{-\lambda \|x\|}.
\end{equation}
\begin{lemma}\label{lem6.2}
 If $\mu$ satisfies {\rm (\ref{intro31})} then $q_{\mu}(x)$ has a continuous derivative at each point of $\R/\Z\setminus\{0\}$ given by
\begin{equation}\label{Sec6.3}
q_{\mu}^{\prime}(x) = \int_0^{\infty} j(\lambda, x)\ \dmu.
\end{equation}
\end{lemma}
\begin{proof} It follows from (\ref{intro31}) and (\ref{Sec6.2.1}) that $\lambda \mapsto j(\lambda,x)$ is integrable with 
respect to $\mu$ at each noninteger point $x$. Assume that $0< \epsilon < \frac{1}{2}$. Then we have
\begin{align}\label{Sec6.4}
\begin{split}
\int_0^{\infty} \int_{\epsilon}^{1-\epsilon} \bigl|j(\lambda,x) \bigr|\ \dy\ \dmu 
	&\le \int_0^{\infty} \int_{\epsilon}^{1-\epsilon} \lambda e^{-\lambda \|y\|}\ \dy\ \dmu \\ 
	& =  2 \int_0^{\infty} \{e^{-\lambda \epsilon} - e^{\lambda/2}\}\ \dmu < \infty.
\end{split}
\end{align}
Assume that $\epsilon \le \|x\|$. Using (\ref{Sec6.2}), (\ref{Sec6.4}) and Fubini's theorem, we obtain the identity
\begin{align}\label{Sec6.5}
\begin{split}
q_{\mu}(x) - q_{\mu}(\tfrac{1}{2})&= \int_0^{\infty} \int_{\tfrac{1}{2}}^{x} j(\lambda,x)\ \dy \ \dmu \\
& =  \int_{\tfrac{1}{2}}^{x} \int_0^{\infty} j(\lambda,x)  \ \dmu\ \dy.
\end{split}
\end{align}
Clearly (\ref{Sec6.5}) implies that $q_{\mu}(x)$ is differentiable on $\R/\Z\setminus\{0\}$ and its derivative is 
given by (\ref{Sec6.3}). Then it follows from (\ref{Sec6.2.1}) and the dominated convergence theorem that $q_{\mu}^{\prime}(x)$ 
is continuous at each point of $\R/\Z\setminus\{0\}$.
\end{proof}

Now assume that $\mu$ satisfies the more restrictive condition (\ref{intro47}).  From (\ref{ef42}) we obtain 
the alternative bound
\begin{equation}\label{ef56}
|p(\lambda, x)| \le \max\big\{\tfrac{2}{\lambda} - \csch\bigl(\tfrac{\lambda}{2}\bigr),
	\coth\bigl(\tfrac{\lambda}{2}\bigr) - \tfrac{2}{\lambda}\big\} 
	= \coth\bigl(\tfrac{\lambda}{2}\bigr) - \tfrac{2}{\lambda}
\end{equation}
at all points $(\lambda, x)$ in $(0,\infty)\times\R/\Z$.  As the function on the right of (\ref{ef56}) is
integrable with respect to $\mu$, it follows from the dominated convergence theorem that
\begin{equation*}\label{ef57}
q_{\mu}(x) = \int_0^{\infty} p(\lambda, x)\ \dmu
\end{equation*}
is continuous on $\R/\Z$.  Also, the Fourier coefficients $\tq_{\mu}(n)$ are nonnegative and satisfy
\begin{align*}\label{ef58}
\sum_{n=-\infty}^{\infty} \tq_{\mu}(n) 
	&= \sum_{\substack{n=-\infty\\n\not= 0}}^{\infty} \int_0^{\infty} \frac{2\lambda}{\lambda^2 + 4\pi^2n^2}\ \dmu \\
	&= \int_0^{\infty} \big\{\coth\bigl(\tfrac{\lambda}{2}\bigr) - \tfrac{2}{\lambda}\big\}\ \dmu < \infty.
\end{align*}
Therefore the partial sums 
\begin{equation}\label{ef59}
q_{\mu}(x) = \lim_{N\rightarrow \infty} \sum_{\substack{n=-N\\ n\not= 0}}^N \tq_{\mu}(n)e(nx)
\end{equation}
converge absolutely and uniformly on $\R/\Z$.

For each nonnegative integer $N$, we define a trigonometric polynomial 
$g_{\mu}(N; x)$, of degree at most $N$, by
\begin{equation}\label{ef60}
g_{\mu}(N; x) = \sum_{n = -N}^N \ug_{\mu}(N; n)e(nx),
\end{equation}
where the Fourier coefficients are given by
\begin{equation}\label{ef61}
\ug_{\mu}(N; 0) = - \int_0^{\infty} \big\{\tfrac{2}{\lambda} 
			- \tfrac{1}{N+1}\csch\bigl(\tfrac{\lambda}{2N+2}\bigr)\big\}\ \dmu,
\end{equation}
and
\begin{equation}\label{ef62}
\ug_{\mu}(N; n) = \tfrac{1}{N+1} \int_0^{\infty} \tL\bigl(\tfrac{\lambda}{N+1}, \tfrac{n}{N+1}\bigr)\ \dmu,
\end{equation}
for $n\not= 0$.  

\begin{theorem}\label{thm7.2}  Assume that $\mu$ satisfies {\rm (\ref{intro31})}.  Then the inequality
\begin{equation}\label{ef63}
g_{\mu}(N; x) \le q_{\mu}(x)
\end{equation}
holds for all $x$ in $\R/\Z$.  If $\wg(x)$ is a real trigonometric polynomial of degree at most $N$ that satisfies
the inequality
\begin{equation}\label{ef64}
\wg(x) \le q_{\mu}(x)
\end{equation}
for all $x$ in $\R/\Z$, then
\begin{equation}\label{ef65}
\int_{\R/\Z} \wg(x)\ \dx \le \int_{\R/\Z} g_{\mu}(N; x)\ \dx.
\end{equation}
Moreover, there is equality in the inequality {\rm (\ref{ef65})} if and only if $\wg(x) = g_{\mu}(N; x)$.
\end{theorem}

\begin{proof}  We will use the elementary identity 
\begin{equation}\label{Sec6.6}
g_{\mu}(N;x) = \int_0^{\infty} l(\lambda,N;x)\ \dmu.
\end{equation}
The inequality on the left hand side of (\ref{ef21}), together with (\ref{ef50}) and (\ref{Sec6.6}), imply (\ref{ef63}). 
Moreover, from (\ref{ef22}) we have
\begin{equation*}
g_{\mu}(N; x) = q_{\mu}(x)
\end{equation*}
for 
\begin{equation*}
x = \tfrac{n - \h}{N+1}\quad\text{and}\quad n = 1, 2, \dots , N+1.
\end{equation*}
The final part of the proof of Theorem \ref{thm7.2} follows as in Theorem \ref{thm6.1}, using the differentiability 
of $q_{\mu}(x)$ on $\R/\Z\setminus\{0\}$ proved in Lemma \ref{lem6.2}.
\end{proof}

If the measure $\mu$ satisfies the more restrictive condition (\ref{intro47}), then we have shown that 
$x\mapsto q_{\mu}(x)$ is continuous on $\R/\Z$, and in particular $q_{\mu}(0)$ is finite. In this case we can exploit 
Theorem \ref{thm1.4} and Theorem \ref{thm6.1} to obtain an extremal trigonometric polynomial of degree at most $N$ that 
majorizes $q_{\mu}(x)$. 

For each nonnegative integer $N$, we define a trigonometric polynomial $h_{\mu}(N; x)$, of degree at most
$N$, by
\begin{equation}\label{ef78}
h_{\mu}(N; x) = \sum_{n = -N}^N \uh_{\mu}(N; n)e(nx),
\end{equation}
where the Fourier coefficients are given by
\begin{equation}\label{ef79}
\uh_{\mu}(N; 0) = \int_0^{\infty} \big\{\tfrac{1}{N+1}\coth\bigl(\tfrac{\lambda}{2N+2}\bigr) - \tfrac{2}{\lambda}\big\}\ \dmu,
\end{equation}
and
\begin{equation}\label{ef80}
\uh_{\mu}(N; n) = \tfrac{1}{N+1} \int_0^{\infty} \tM\bigl(\tfrac{\lambda}{N+1}, \tfrac{n}{N+1}\bigr)\ \dmu,
\end{equation}
for $n\not= 0$.  The proof of the following result is similar to the proof of Theorem \ref{thm7.2}

\begin{theorem}\label{thm7.3}  Assume that $\mu$ satisfies {\rm (\ref{intro47})}.  Then the inequality
\begin{equation}\label{ef81}
q_{\mu}(x) \le h_{\mu}(N; x)
\end{equation}
holds for all $x$ in $\R/\Z$.  If $\wh(x)$ is a real trigonometric polynomial of degree at most $N$ that satisfies
the inequality
\begin{equation}\label{ef82}
q_{\mu}(x) \le \wh(x)
\end{equation}
for all $x$ in $\R/\Z$, then
\begin{equation}\label{ef83}
\int_{\R/\Z} h_{\mu}(N; x)\ \dx \le \int_{\R/\Z} \wh(x)\ \dx.
\end{equation}
Moreover, there is equality in the inequality {\rm (\ref{ef83})} if and only if $\wh(x) = h_{\mu}(N; x)$.
\end{theorem}
We note that Theorem \ref{thm1.5}, described in the introduction of this paper, is a 
special case of Theorem \ref{thm7.2} when applied to the Haar measure $\mu$ defined in (\ref{intro80}). For 
this it is sufficient to compare the Fourier coefficients
\begin{equation}
\tq_{\mu}(n) = \int_0^{\infty} \frac{2\lambda}{\lambda^2 + 4\pi^2n^2}\ \lambda^{-1} \dl = \dfrac{1}{2|n|}, \ \ n\neq 0,
\end{equation}
given by (\ref{ef54}), with the well known Fourier expansion
\begin{equation}
-\log \bigl|1 - e(x) \bigr| = - \log \bigl| 2\sin \pi x \bigr| = \sum_{n\neq 0} \dfrac{1}{2|n|}e(nx).
\end{equation}
We define therefore $u_{N}(x) = -g_{\mu}(N;x)$. Equality (\ref{intro92}) follows from (\ref{ef61}) and
\begin{equation}
\widehat{u}_N(0) =  \int_0^{\infty} \big\{\tfrac{2}{\lambda} 
			- \tfrac{1}{N+1}\csch\bigl(\tfrac{\lambda}{2N+2}\bigr)\big\}\ \lambda^{-1}\ \dl = \dfrac{\log 2}{N+1}\,.
\end{equation}
Finally, the bound (\ref{intro93}) follows from (\ref{ef62}) and Corollary \ref{cor3.1}.

\section{Bounds for Hermitian forms}

Let $\mu$ be a measure on the Borel subsets of $(0,\infty)$ that satisfies (\ref{intro31}).  Define the function 
$r_{\mu}:\R\rightarrow [0,\infty]$ by
\begin{equation}\label{hf1}
r_{\mu}(t) = \int_0^{\infty} \frac{2\lambda}{\lambda^2 + 4\pi^2 t^2}\ \dmu.
\end{equation}
It follows using (\ref{intro31}) that $r_{\mu}(t)$ is even, continuous, finite for 
all $t\not= 0$, and nonincreasing for $0 < t$.  

Let $\xi_0, \xi_1, \xi_2, \dots , \xi_N$ be distinct real numbers such that $0 < \delta \le |\xi_m - \xi_n|$ whenever 
$m\not= n$.  We consider the Hermitian form defined for vectors $\ba$ in $\C^{N+1}$ by
\begin{equation}\label{hf2}
\ba\mapsto \sum_{m=0}^N\sum_{\substack{n=0\\n\not= m}}^N a_m\overline{a}_n r_{\mu}(\xi_m - \xi_n),
\end{equation}
where $\overline{a}_n$ is the complex conjugate of $a_n$.  

\begin{theorem}\label{thm7.1}  If $\mu$ satisfies {\rm (\ref{intro31})} then
\begin{equation}\label{bd1}
- A(\delta, \mu) \sum_{n=0}^N \bigl|a_n\bigr|^2 
	\le \sum_{m=0}^N\sum_{\substack{n=0\\n\not= m}}^N a_m\overline{a}_n r_{\mu}(\xi_m - \xi_n),
\end{equation}
for all complex numbers $a_n$, where
\begin{equation}\label{bd2}
A(\delta, \mu) 
  = \int_0^{\infty} \big\{\tfrac{2}{\lambda} 
			- \tfrac{1}{\delta}\csch\left(\tfrac{\lambda}{2\delta}\right)\big\}\ \dmu.
\end{equation}
The inequality {\rm (\ref{bd1})} is sharp in the sense that the positive constant $A(\delta, \mu)$ defined by {\rm (\ref{bd2})}
cannot be replaced by a smaller number.

If $\mu$ satisfies {\rm (\ref{intro47})} then
\begin{equation}\label{bd3}
\sum_{m=0}^N\sum_{\substack{n=0\\n\not= m}}^N a_m\overline{a}_n r_{\mu}(\xi_m - \xi_n)
	\le B(\delta, \mu) \sum_{n=0}^N \bigl|a_n\bigr|^2
\end{equation}
for all complex numbers $a_n$, where
\begin{equation}\label{bd4}
B(\delta, \mu) 
  = \int_0^{\infty} \big\{\tfrac{1}{\delta}\coth\left(\tfrac{\lambda}{2\delta}\right) - \tfrac{2}{\lambda}\big\}\ \dmu.
\end{equation}
The inequality {\rm (\ref{bd3})} is sharp in the sense that the positive constant $B(\delta, \mu)$ defined by {\rm (\ref{bd4})}
cannot be replaced by a smaller number.
\end{theorem}

\begin{proof}
Write
\begin{equation*}\label{hf3}
u(x) = f_{\mu}(x) - f_{\mu}\bigl(\delta^{-1}\bigr) - G_{\nu}(\delta x)
\end{equation*}
for the nonnegative, integrable function that occurs in the statement of Theorem \ref{thm1.3}.  Then we have
\begin{align}\label{hf4}
\begin{split}
0 &\le \int_{-\infty}^{\infty} u(x)\Bigl|\sum_{m=0}^N a_m e(-\xi_m x)\Bigr|^2\ \dx \\
  &= \sum_{m=0}^N\sum_{n=0}^N a_m\overline{a}_n \int_{-\infty}^{\infty} u(x) e\bigl((\xi_n - \xi_m) x\bigr)\ \dx \\
  &= \tu(0) \sum_{n=0}^N \bigl|a_n\bigr|^2 
		+ \sum_{m=0}^N\sum_{\substack{n=0\\n\not= m}}^N a_m\overline{a}_n \tu(\xi_m - \xi_n).
\end{split}
\end{align}
As $\delta \le |\xi_m - \xi_n|$ whenever $m\not= n$, we get
\begin{equation*}\label{hf5}
\tu(\xi_m - \xi_n) = r_{\mu}(\xi_m - \xi_n)
\end{equation*}
by (\ref{intro64}) and (\ref{hf1}).  Thus (\ref{intro63}) and (\ref{hf4}) lead to the lower bound
\begin{equation}\label{hf6}
- A(\delta, \mu) \sum_{n=0}^N \bigl|a_n\bigr|^2 
	\le \sum_{m=0}^N\sum_{\substack{n=0\\n\not= m}}^N a_m\overline{a}_n r_{\mu}(\xi_m - \xi_n),
\end{equation}
where we have written
\begin{equation}\label{hf7}
A(\delta, \mu) = \tu(0) = \int_0^{\infty} \big\{\tfrac{2}{\lambda} 
			- \tfrac{1}{\delta}\csch\left(\tfrac{\lambda}{2\delta}\right)\big\}\ \dmu.
\end{equation}

Let $\nu$ be the measure defined on Borel subsets $E\subseteq (0,\infty)$ by (\ref{intro59}).  It follows from (\ref{hf1}) that
\begin{equation*}\label{bd5}
r_{\mu}(\delta t) = \delta^{-1}r_{\nu}(t)
\end{equation*}
for all real $t\not= 0$.  For $0 < x < 1$ we use (\ref{ef55}) and obtain the identity
\begin{align}\label{bd6} 
\begin{split}
\lim_{N\rightarrow \infty} \sum_{\substack{n=-N\\ n\not= 0}}^N r_{\mu}(\delta n)e(nx)
	&= \lim_{N\rightarrow \infty} \delta^{-1} \sum_{\substack{n=-N\\ n\not= 0}}^N r_{\nu}(n)e(nx) \\
	&= \lim_{N\rightarrow \infty} \delta^{-1} \sum_{\substack{n=-N\\ n\not= 0}}^N \tq_{\nu}(n)e(nx) \\
	&= \delta^{-1} \int_0^{\infty} p(\lambda, x)\ \dnu.
\end{split}
\end{align}
In particular, at $x = \h$ we find that
\begin{align}\label{bd7}
\begin{split}
\lim_{N\rightarrow \infty} \sum_{\substack{n=-N\\ n\not= 0}}^N (-1)^n r_{\mu}(\delta n)
	&= \delta^{-1} \int_0^{\infty} p(\lambda, \hh)\ \dnu \\
	&= - \int_0^{\infty} \big\{\tfrac{2}{\lambda} 
			- \tfrac{1}{\delta}\csch\left(\tfrac{\lambda}{2\delta}\right)\big\}\ \dmu.
\end{split}
\end{align}
To see that the constant $A(\delta, \mu)$ is sharp we apply (\ref{hf6}) with 
\begin{equation*}\label{bd8}
a_n = (N+1)^{-1/2} (-1)^n,\quad\text{and}\quad \xi_n = \delta n.
\end{equation*}
We find that 
\begin{align}\label{bd9}
\begin{split}
- A(\delta, \mu)
	&\le (N+1)^{-1} \sum_{m=0}^N\sum_{\substack{n=0\\n\not= m}}^N (-1)^{m-n} r_{\mu}\bigl(\delta(m-n)\bigr) \\
	&= (N+1)^{-1} \sum_{\substack{n=-N\\n\not= 0}}^N (N+1 - |n|) (-1)^n r_{\mu}(\delta n).
\end{split}
\end{align}
We let $N\rightarrow \infty$ on the right hand side of (\ref{bd9}) and use (\ref{bd7}).  In this way we conclude that
\begin{equation*}\label{bd10}
\int_0^{\infty} \big\{\tfrac{2}{\lambda} 
			- \tfrac{1}{\delta}\csch\left(\tfrac{\lambda}{2\delta}\right)\big\}\ \dmu \le A(\delta, \mu).
\end{equation*}

Now suppose that $\mu$ satisfies the more restrictive condition (\ref{intro47}).  Write
\begin{equation*}\label{hf8}
v(x) = H_{\nu}(\delta x) + f_{\mu}\bigl(\delta^{-1}\bigr) - f_{\mu}(x)
\end{equation*} 
for the nonnegative, integrable function that occurs in the statement of
Theorem \ref{thm1.4}.  We proceed as in (\ref{hf4}) to derive the inequality
\begin{align}\label{hf9}
\begin{split}
0 &\le \int_{-\infty}^{\infty} v(x)\Bigl|\sum_{m=0}^N a_m e(-\xi_m x)\Bigr|^2\ \dx \\
  &= \tv(0) \sum_{n=0}^N \bigl|a_n\bigr|^2 
  		+ \sum_{m=0}^N\sum_{\substack{n=0\\n\not= m}}^N a_m\overline{a}_n \tv(\xi_m - \xi_n).
\end{split}
\end{align}
In this case (\ref{intro73}) and (\ref{hf1}) imply that
\begin{equation*}\label{hf9.5}
\tv(\xi_m - \xi_n) = -r_{\mu}(\xi_m - \xi_n)
\end{equation*}
whenever $m\not= n$.  Therefore (\ref{intro72}) and (\ref{hf9}) lead to the upper bound
\begin{equation}\label{hf10}
\sum_{m=0}^N\sum_{\substack{n=0\\n\not= m}}^N a_m\overline{a}_n r_{\mu}(\xi_m - \xi_n)
	\le B(\delta, \mu) \sum_{n=0}^N \bigl|a_n\bigr|^2
\end{equation}
where
\begin{equation}\label{hf11}
B(\delta, \mu) = \tv(0)
  = \int_0^{\infty} \big\{\tfrac{1}{\delta}\coth\left(\tfrac{\lambda}{2\delta}\right) - \tfrac{2}{\lambda}\big\}\ \dmu.
\end{equation}

If $\mu$ satisfies (\ref{intro47}) then (\ref{ef59}) holds for all $x$ in $\R/\Z$.  Thus the identity
(\ref{bd6}) continues to hold.  In particular, at $x = 0$ we find that
\begin{align}\label{bd12}
\begin{split}
\lim_{N\rightarrow \infty} \sum_{\substack{n=-N\\ n\not= 0}}^N r_{\mu}(\delta n)
  &= \delta^{-1} \int_0^{\infty} p(\lambda, 0)\ \dnu \\
  &= \int_0^{\infty} \big\{\tfrac{1}{\delta}\coth\left(\tfrac{\lambda}{2\delta}\right) - \tfrac{2}{\lambda}\big\}\ \dmu.
\end{split}
\end{align} 
To show that the constant $B(\delta, \mu)$ is sharp we apply (\ref{hf10}) with
\begin{equation*}\label{bd13}
a_n = (N+1)^{-1/2},\quad\text{and}\quad \xi_n = \delta n.
\end{equation*}
In this case we find that
\begin{align}\label{bd14}
(N+1)^{-1} \sum_{\substack{n=-N\\n\not= 0}}^N (N+1 - |n|) r_{\mu}(\delta n) \le B(\delta, \mu). 
\end{align}
We let $N\rightarrow \infty$ on the left of $(\ref{bd14})$ and use (\ref{bd12}).  We conclude that
\begin{equation*}\label{bd15}
\int_0^{\infty} \big\{\tfrac{1}{\delta}\coth\left(\tfrac{\lambda}{2\delta}\right) 
	- \tfrac{2}{\lambda}\big\}\ \dmu \le B(\delta, \mu).
\end{equation*}
This proves the theorem.
\end{proof}

An interesting special case of the Hermitian forms considered here occurs by selecting the measure $\mu_{\sigma}$ defined in (\ref{Intro26.1}). We recall that for $0 < \sigma < 2$ the measure $\mu_{\sigma}$ satisfies the 
condition (\ref{intro31}), and it satisfies (\ref{intro47}) only for $1 < \sigma < 2$. For this special case we obtain the following inequalities, which are related to the discrete one dimensional Hardy-Littlewood-Sobolev inequalities (see \cite[page 288]{HPL}).

\begin{corollary}\label{cor7.2}  Let $\xi_0, \xi_1, \xi_2, \dots , \xi_N$ be real numbers such that 
$0 < \delta \le |\xi_m - \xi_n|$ whenever $m\not= n$.  Let $a_0, a_1, a_2, \dots , a_N$ be complex numbers.
If $0 < \sigma < 1$ then
\begin{equation}\label{hf20}
-\frac{(2 - 2^{2 - \sigma})\zeta(\sigma)}{\delta^{\sigma}} \sum_{n=0}^N |a_n|^2
	\le \sum_{m=0}^N\sum_{\substack{n=0\\n\not= m}}^N \frac{a_m\overline{a}_n}{|\xi_m - \xi_n|^{\sigma}},
\end{equation}
if $\sigma = 1$ then
\begin{equation}\label{hf21}
-\frac{\log 4}{\delta} \sum_{n=0}^N |a_n|^2
	\le \sum_{m=0}^N\sum_{\substack{n=0\\n\not= m}}^N \frac{a_m\overline{a}_n}{|\xi_m - \xi_n|},
\end{equation}
and if $1 < \sigma < 2$ then
\begin{equation}\label{hf22}
-\frac{(2 - 2^{2 - \sigma})\zeta(\sigma)}{\delta^{\sigma}} \sum_{n=0}^N |a_n|^2
	\le \sum_{m=0}^N\sum_{\substack{n=0\\n\not= m}}^N \frac{a_m\overline{a}_n}{|\xi_m - \xi_n|^{\sigma}}
	\le \frac{2\zeta(\sigma)}{\delta^{\sigma}} \sum_{n=0}^N |a_n|^2,
\end{equation}
where $\zeta$ denotes the Riemann zeta-function.  The constants occurring in these inequalities are sharp.
\end{corollary}

\begin{proof}  For $\sigma\not= 1$ the integral on the right of (\ref{hf7}) is given by
\begin{equation*}\label{hf23}
\int_0^{\infty} \big\{\tfrac{2}{\lambda} - 
	\tfrac{1}{\delta}\csch\bigl(\tfrac{\lambda}{2\delta}\bigr)\big\}\ \lambda^{-\sigma}\ \dl
	= \frac{\bigl(2 - 2^{2 - \sigma}\bigr)\Gamma(1 - \sigma)\zeta(1 - \sigma)}{\delta^{\sigma}},
\end{equation*}
where $\zeta$ is the Riemann zeta-function.  And for $0 < \sigma < 2$ we find that
\begin{equation*}\label{hf24}
\int_0^{\infty} \frac{2\lambda}{\lambda^2 + 4\pi^2 t^2}\ \lambda^{- \sigma}\ \dl
	= \frac{\pi}{(2\pi |t|)^{\sigma}\sin \tfrac{\pi\sigma}{2}}.
\end{equation*}
When these identities are used in (\ref{hf6}) we obtain the inequality
\begin{align}\label{hf25}
\begin{split}
\frac{\bigl(2 - 2^{2 - \sigma}\bigr)\Gamma(1 - \sigma)\zeta(1 - \sigma)}{\delta^{\sigma}}&\sum_{n=0}^N |a_n|^2 \\
		\le \frac{\pi}{(2\pi)^{\sigma}\sin \tfrac{\pi\sigma}{2}}
		&\sum_{m=0}^N\sum_{\substack{n=0\\n\not= m}}^N \frac{a_m\overline{a}_n}{|\xi_m - \xi_n|^{\sigma}}.
\end{split}
\end{align}
Then (\ref{hf25}) leads to the lower bounds in (\ref{hf20}) and (\ref{hf22}) by using the functional equation for the
Riemann zeta-function.  

If $\sigma = 1$ we have
\begin{equation}\label{hf26}
\int_0^{\infty} \big\{\tfrac{2}{\lambda} 
	- \tfrac{1}{\delta}\csch\bigl(\tfrac{\lambda}{2\delta}\bigr)\big\}\ \lambda^{-1}\ \dl = \frac{\log 2}{\delta},
\end{equation}
and 
\begin{equation}\label{hf27}
\int_0^{\infty} \frac{2\lambda}{\lambda^2 + 4\pi^2 t^2}\ \lambda^{-1}\ \dl = \frac{1}{2|t|}.
\end{equation}
We use (\ref{hf26}) and (\ref{hf27}) in (\ref{hf6}) and obtain the remaining lower bound (\ref{hf21}).

For $1 < \sigma < 2$ the integral on the right of (\ref{hf11}) is
\begin{equation*}\label{hf28}
\int_0^{\infty} \big\{\tfrac{1}{\delta}
	\coth\left(\tfrac{\lambda}{2\delta}\right) - \tfrac{2}{\lambda}\big\}\ \lambda^{-\sigma}\ \dl
	= \frac{2\Gamma(1 - \sigma)\zeta(1 - \sigma)}{\delta^{\sigma}}.
\end{equation*}
\end{proof}
We can extend the inequality (\ref{hf22}) to the case $\sigma = 2$ by continuity. A natural question is 
whether the inequality (\ref{hf22}) remains valid for $\sigma >2$.  F. Littmann showed in \cite{Lit} that 
this true when $\sigma$ is an even integer, which suggests an affirmative answer. We expect to return to this subject in a future paper.

\section{Erd\"os-Tur\'an Inequalities}

Let $x_1, x_2, \dots , x_M$ be a finite set of points in $\R/\Z$.  A basic problem in the theory
of equidistribution is to estimate the discrepancy of the points $x_1, x_2, \dots , x_M$ by an expression
that depends on the Weyl sums
\begin{equation}\label{et0}
\sum_{m=1}^M e(nx_m),\quad\text{where}\quad n = 1, 2, \dots , N.
\end{equation} 
This is most easily accomplished by introducing the sawtooth function $\psi:\R/\Z\rightarrow \R$, defined by
\begin{equation*}\label{et1}
\psi(x) = \begin{cases} x-[x]-\tfrac12 & \text{if $x$ is not in $\Z$},\\
            0 & \text{if $x$ is in $\Z$},\end{cases}
\end{equation*}
where $[x]$ is the integer part of $x$.  Then a simple definition for the discrepancy of the finite set is 
\begin{equation*}\label{et2}
D_M(\bx) = \sup_{y\in\R/\Z} \Bigl|\sum_{m=1}^M \psi\bigl(x_m - y\bigr)\Bigr|.
\end{equation*}
In this setting the Erd\"{o}s-Tur\'{a}n inequality is an upper bound for $D_M$ of the form
\begin{equation}\label{et3}
D_M(\bx) \le c_1 MN^{-1} + c_2 \sum_{n=1}^N n^{-1} \Bigl|\sum_{m=1}^M e(nx_m)\Bigr|,
\end{equation}
where $c_1$ and $c_2$ are positive constants.  In applications to specific sets the parameter $N$ can be 
selected so as to minimize the right hand side of (\ref{et3}).  Bounds of this kind follow easily from knowledge 
of the extremal trigonometric polynomials that majorize and minorize the function $\psi(x)$.  This is discussed in
\cite{ET}, \cite{M2}, \cite{V}, and \cite{V2}.  An extension to the spherical cap discrepancy is derived 
in \cite{LV}, and a related inequality in several variables is obtained in \cite{BMV}.

Let $F_M(z)$ be the monic polynomial in $\C[z]$ having roots on the unit circle at the points 
$e(x_1), e(x_2), \dots , e(x_M)$, so that
\begin{equation*}\label{et4}
F_M(z) = \prod_{m=1}^M \bigl(z - e(x_m)\bigr).
\end{equation*}
Then an alternative expression, which also measures the relative uniform distribution of the points 
$x_1, x_2, \dots , x_M$ in $\R/\Z$, is given by
\begin{equation*}\label{et5}
\sup_{|z|\le 1} \log \bigl|F_M(z)\bigr| = \sup_{y\in\R/\Z} \sum_{m=1}^M \log \bigl|1 - e(x_m - y)\bigr|.
\end{equation*}
Using Theorem \ref{thm1.5} we obtain the bound
\begin{align}\label{et6}
\begin{split}
\sum_{m=1}^M \log \bigl|1 - e(x_m - y)\bigr| &\le \sum_{m=1}^M u_N(x_m - y) \\
	&= M(N+1)^{-1}\log 2 \\
	&\qquad\qquad + \sum_{1 \le |n| \le N} \tu_N(n) \Big\{\sum_{m=1}^M e(nx_m)\Big\} e(-ny) \\
	&\le M(N+1)^{-1}\log 2 + \sum_{n=1}^N n^{-1} \Bigl|\sum_{m=1}^M e(nx_m)\Bigr|,
\end{split}
\end{align}
which is analogous to (\ref{et3}).  We establish a generalization of this bound to polynomials with
zeros not necessarily on the unit circle.

Let $\alpha_1, \alpha_2, ..., \alpha_M $ be complex numbers and define
\begin{equation}\label{et7}
F_M(z) = \prod_{m=1}^M (z - \alpha_m).
\end{equation}
We wish to estimate $\sup \{|F_M(z)|: |z| \le 1\}$ by an expression that depends on the power sums
\begin{equation}\label{et8}
\sum_{m=1}^{M} (\alpha_m)^n,\quad\text{where}\quad 1\leq n \leq N.
\end{equation}

\begin{theorem}\label{thm8.1}  Let $F_M(z)$ be the monic polynomial defined by {\rm (\ref{et7})} and assume that 
$|\alpha_m| \le 1$ for each $m = 1, 2, \dots , M$.  Then for each nonnegative integer $N$ we have
\begin{equation}\label{et9}
\sup_{|z| \le 1} \log |F_M(z)| \le M(N+1)^{-1}\log 2 + \sum_{n=1}^N n^{-1} \Bigl|\sum_{m=1}^M (\alpha_m)^n\Bigr|.
\end{equation}
\end{theorem}

\begin{proof}  Let $u_N(x)$ be the trigonometric polynomial that occurs in Theorem \ref{thm1.5}, and let $v_N(z)$ denote 
the algebraic polynomial
\begin{equation}\label{et10}
v_N(z) = \tu_N(0) + 2 \sum_{n=1}^N \tu_N(n) z^n.
\end{equation}
Then (\ref{intro91}) can be extended to the inequality
\begin{equation}\label{intro97}
\log |1 - z| \le \Re\{v_N(z)\}
\end{equation}
for all complex numbers $z$ with $|z|\le 1$.  This follows from the observation that both sides of (\ref{intro97}) are harmonic
functions on the open unit disk $\Delta = \{z\in\C: |z|<1\}$, and (\ref{intro91}) asserts that (\ref{intro97}) holds at each
point $z = e(x)$ on the boundary of $\Delta$.

As $z\mapsto \log |F_M(z)|$ is subharmonic on $\Delta$, there exists a point $e(y)$ on the boundary of $\Delta$ such that
\begin{align}\label{et11}
\begin{split}
\sup_{|z| \le 1} \log |F_M(z)| &= \log \bigl|F_M\bigl(e(y)\bigr)\bigr| \\
	&= \sum_{m=1}^M \log |1 - e(-y)\alpha_m| \\
	&\le \sum_{m=1}^M \Re\big\{v_N\bigl(e(-y)\alpha_m\bigr)\big\} \\
	&= M\tu_N(0) + 2 \sum_{n=1}^N \tu_N(n) \Re\Big\{e(-ny) \sum_{m=1}^M (\alpha_m)^n\Big\}.
\end{split}
\end{align}
The inequality (\ref{et9}) follows from (\ref{et11}) by applying (\ref{intro92}) and (\ref{intro93}).
\end{proof}

If $F_M(z)$ is defined by (\ref{et7}), but we do not assume that the roots are in the closed unit disk, we can
still obtain a bound for $\sup \{|F_M(z)|: |z| \le 1\}$.  In this more general case, however, we must modify the
power sums (\ref{et8}).  Suppose that the roots of $F_M$ are arranged so that
\begin{equation*}\label{et20}
0 \le |\alpha_1| \leq |\alpha_2| \leq \dots \leq |\alpha_L| \leq 1 < |\alpha_{L+1}| \leq \dots \leq |\alpha_{M}|. 
\end{equation*}
Then define
\begin{equation}\label{et21}
\beta_m = \begin{cases} \alpha_m & \text{if $1\leq m \leq L$,} \\
		(\overline{\alpha}_m)^{-1} & \text{if $L+1 \leq m \leq M$,}\end{cases}
\end{equation}

\begin{corollary}\label{cor8.2}  Let $F_M(z)$ be the monic polynomial defined by {\rm (\ref{et7})} and let
\begin{equation*}
\beta_1, \beta_2, \dots , \beta_M 
\end{equation*}
be complex numbers defined by {\rm (\ref{et21})}.  Then for each nonnegative 
integer $N$ we have
\begin{align}\label{et22}
\begin{split}
\sup_{|z| \le 1} \log |F_M(z)| \le \sum_{m=1}^M &\log^+ |\alpha_m| \\
	&+ M(N+1)^{-1}\log 2 + \sum_{n=1}^N n^{-1} \Bigl|\sum_{m=1}^M (\beta_m)^n\Bigr|.
\end{split}
\end{align}
\end{corollary}

\begin{proof}  Define the finite Blaschke product
\begin{equation*}\label{et23}
B(z) = \prod_{l=L+1}^M \frac{1 - \overline{\alpha}_lz}{z - \alpha_l},
\end{equation*}
so that if $|z| = 1$ then $|B(z)| = 1$.  We find that
\begin{equation*}\label{et24}
G_M(z) = B(z)F_M(z) = \prod_{l=L+1}^M(-\overline{\alpha}_l) \prod_{m=1}^M (z - \beta_m),
\end{equation*}
is a polynomial with roots $\beta_1, \beta_2, \dots , \beta_M$.  As $|\beta_m| \le 1$ for each $m = 1, 2, \dots , M$,
we apply Theorem \ref{thm8.1} to $G_M(z)$ and (\ref{et22}) follows immediately.
\end{proof}

We note that by Jensen's formula the first sum on the right of (\ref{et22}) is
\begin{equation*}\label{et25}
\sum_{m=1}^M \log^+ |\alpha_m| = \int_{\R/\Z} \log \bigl|F_M\bigl(e(x)\bigr)\bigr|\ \dx.
\end{equation*}
Therefore this sum by itself could not be an upper bound for the left hand side of (\ref{et22}), except in the trivial
case where $\alpha_1 = \cdots = \alpha_m = 0$.


\begin{thebibliography}{99}

\bibitem{BMV}
J.~T.~Barton, H.~L.~Montgomery, and J.~D.~Vaaler,
\newblock Note on a Diophantine inequality in several variables,
\newblock Proc. Amer. Math. Soc., 129, (2000), 337--345.

\bibitem{ET} 
P.~Erd\"{o}s and P.~Tur\'{a}n,
\newblock On a problem in the theory of uniform distribution,
\newblock Indag. Math., 10, (1948), 370--378.

\bibitem{HPL} 
G.~Hardy, G.~Polya and J.~Littlewood,
\newblock {\em Inequalities},
\newblock Cambridge University Press, 1967.

\bibitem{HV} 
J.~Holt and J.~D.~Vaaler,
\newblock The Beurling-Selberg extremal functions for a ball in the Euclidean space,
\newblock Duke Mathematical Journal, 83, (1996), 203--247.

\bibitem{GV}
S.~W.~Graham and J.~D.~Vaaler,
\newblock A class of Extremal Functions for the Fourier Transform,
\newblock Tran. Amer. Math. Soc. 265, (1981), 283--302.

\bibitem{GV2}
S.~W.~Graham and J.~D.~Vaaler,
\newblock Extremal functions for the Fourier transform and the large sieve,
\newblock Topics in Classical Number Theory, Vol. I, II (Budapest, 1981),  
\newblock Colloq. Math. Soc. János Bolyai, 34, North-Holland, Amsterdam, 599--615.

\bibitem{L}
M.~Lerma,
\newblock {\em An extremal majorant for the logarithm and its applications},
\newblock Ph.D. Dissertation, Univ. of Texas at Austin, 1998.

\bibitem{LV} 
X.~J.~Li and J.~D.~Vaaler,
\newblock Some trigonometric extremal functions and the Erd\"{o}s-Tur\'{a}n type inequalities,
\newblock Indiana Univ. Math. J. 48, (1999), no. 1, 183--236.

\bibitem{Lit}
F.~Littmann,
\newblock Entire majorants via Euler-Maclaurin summation,
\newblock Trans. Amer. Math. Soc. 358, (2006), no. 7, 2821--2836.

\bibitem{Lit2} 
F.~Littmann,
\newblock Entire approximations to the truncated powers,
\newblock Constr. Approx. 22 (2005), no. 2, 273--295.

\bibitem{M} 
H.~L.~Montgomery,
\newblock The analytic principle of the large sieve,
\newblock Bull. Amer. Math. Soc. 84, (1978), no. 4, 547--567.

\bibitem{M2}
H.~L.~Montgomery,
\newblock {\em Ten Lectures on the Interface Between Analytic Number Theory and Harmonic Analysis},
\newblock CBMS No. 84, Amer. Math. Soc., Providence, 1994.

\bibitem{MV} 
H.~L.~Montgomery and R.~C.~Vaughan,
\newblock Hilbert's Inequality,
\newblock J. London Math. Soc. (2), 8, (1974), 73--81.

\bibitem{PP}
M.~Plancherel and G.~Polya,
\newblock Fonctions enti\'eres et int\'egrales de Fourier multiples, (Seconde partie)
\newblock Comment. Math. Helv. 10, (1938), 110--163.

\bibitem{Rudin}
W.~Rudin,
\newblock {\em Real and Complex Analysis}, 3rd edition,
\newblock McGraw-Hill, New York, 1987

\bibitem{S}
A.~Selberg,
\newblock Lectures on Sieves, {\em Atle Selberg: Collected Papers}, Vol. II,
\newblock Springer-Verlag, Berlin, 1991, pp. 65--247.

\bibitem{V}
J.~D.~Vaaler,
\newblock Some extremal functions in Fourier analysis,
\newblock Bull. Amer. Math. Soc. 12, (1985), 183--215.

\bibitem{V2}
J.~D.~Vaaler,
\newblock Refinements of the Erd\"{o}s-Tur\'{a}n inequality,
\newblock {\em Number Theory with an emphasis on the Markoff spectrum} 
\newblock (Provo, UT, 1991), Lecture Notes in Mathematics, no. 147, 
\newblock Dekker, 1993, 263--269.

\bibitem{Y}
R.~M.~Young,
\newblock {\em An Introduction to Nonharmonic Fourier Series},
\newblock Academic Press, New York, 1980.

\bibitem{Z} 
A.~Zygmund,
\newblock {\em Trigonometric Series},
\newblock Cambridge University Press, 1959.

\end{thebibliography}
\end{document}